\documentclass[a4paper,reqno]{amsart}

\usepackage{amssymb}
\usepackage{amstext}
\usepackage{amsmath}
\usepackage{amscd}
\usepackage{amsthm}
\usepackage{amsfonts}
\usepackage{enumerate}
\usepackage{graphicx}
\usepackage{latexsym}
\usepackage{mathrsfs}
\usepackage{mathtools}
\usepackage[all]{xy}
\xyoption{all}

\usepackage{pstricks}
\usepackage{lscape}
\usepackage{comment}

\newtheorem{theorem}{Theorem}[section]

\newtheorem{corollary}[theorem]{Corollary}
\newtheorem{lemma}[theorem]{Lemma}
\newtheorem{proposition}[theorem]{Proposition}
\newtheorem{definition-proposition}[theorem]{Definition-Proposition}

\theoremstyle{definition}
\newtheorem{definition}[theorem]{Definition}

\newtheorem{remark}[theorem]{Remark}
\newtheorem{example}[theorem]{Example}

\newcommand{\Ext}{\operatorname{Ext}\nolimits}
\newcommand{\Tor}{\operatorname{Tor}\nolimits}
\newcommand{\Hom}{\operatorname{Hom}\nolimits}
\newcommand{\End}{\operatorname{End}\nolimits}
\newcommand{\Tr}{\operatorname{Tr}\nolimits}
\renewcommand{\mod}{\mathsf{mod}\hspace{.01in}}
\newcommand{\Mod}{\mathsf{Mod}\hspace{.01in}}
\newcommand{\add}{\mathsf{add}\hspace{.01in}}
\newcommand{\Fac}{\mathsf{Fac}\hspace{.01in}}
\newcommand{\tilt}{\mbox{\rm tilt}\hspace{.01in}}
\newcommand{\sttilt}{\mbox{\rm s$\tau$-tilt}\hspace{.01in}}
\newcommand{\op}{\operatorname{op}\nolimits}
\newcommand{\bo}{\operatorname{b}\nolimits}
\newcommand{\DD}{\mathsf{D}}
\newcommand{\TT}{\mathcal{T}}

\newcommand{\rad}{\operatorname{rad}\nolimits}
\newcommand{\soc}{\operatorname{soc}\nolimits}
\newcommand{\projdim}{\mathop{{\rm proj.dim}}\hspace{.01in}}
\newcommand{\injdim}{\mathop{{\rm inj.dim}}\hspace{.01in}}
\newcommand{\gldim}{\mathop{{\rm gl.dim}}\hspace{.01in}}
\newcommand{\RHom}{\mathbf{R}\strut\kern-.2em\operatorname{Hom}\nolimits}
\newcommand{\LotimesL}{\mathop{\otimes^{\mathbf{L}}_\Lambda}\nolimits}

\numberwithin{equation}{section}

\usepackage{paralist}

\def\Tr{\mathop{\rm Tr}\nolimits}
\def\rad{\mathop{\rm rad}\nolimits}

\def\gldim{\mathop{\rm gl.dim}\nolimits}

\def\dim{\mathop{\rm dim}\nolimits}

\begin{document}
\title{Classifying $\tau$-tilting modules over the Auslander algebra of $K[x]/(x^n)$}

\thanks{The first author is supported
by JSPS Grant-in-Aid for Scientific Research (B) 24340004, (C)
23540045 and (S) 15H05738. The second author is supported by NSFC
(Nos.11101217, 11401488 and 11571164) and Jiangsu Government
Scholarship for Overseas Studies (JS-2014-352). }

\author{Osamu Iyama}
\address{O. Iyama: Graduate School of Mathematics, Nagoya
University, Nagoya, 464-8602, Japan}
\email{iyama@math.nagoya-u.ac.jp}

%\author{Yusuke Tsujioka}

\author{Xiaojin Zhang}
\address{X. Zhang: School of Mathematics and Statistics, NUIST,
Nanjing, 210044, P. R. China} \email{xjzhang@nuist.edu.cn}
\maketitle

\begin{abstract} We build a bijection between the set $\sttilt\Lambda$ of isomorphism classes of basic support
$\tau$-tilting modules over the Auslander algebra $\Lambda$ of $K[x]/(x^n)$ and the symmetric group $\mathfrak{S}_{n+1}$,
which is an anti-isomorphism of partially ordered sets with respect to the generation order on $\sttilt\Lambda$ and the left order on $\mathfrak{S}_{n+1}$.
This restricts to the bijection between the set $\tilt\Lambda$ of isomorphism classes of basic tilting $\Lambda$-modules and
the symmetric group $\mathfrak{S}_n$ due to Br\"{u}stle, Hille, Ringel and R\"{o}hrle.
Regarding the preprojective algebra $\Gamma$ of Dynkin type $A_n$ as a factor algebra of $\Lambda$,
we show that the tensor functor $-\otimes_{\Lambda}\Gamma$ induces a bijection between $\sttilt\Lambda\to\sttilt\Gamma$.
This recover Mizuno's anti-isomorphism $\mathfrak{S}_{n+1}\to\sttilt\Gamma$ of posets for type $A_n$.
\end{abstract}

\tableofcontents

\section{Introduction}

Tilting theory has been central in the representation theory of
finite dimensional algebras since the early seventies \cite{BGP,
AuPR, B, BrB, HaR}. In this theory, tilting modules play a central
role. So it is important to classify tilting modules for a given
algebra. There are many algebraists working on this topic which
makes the theory fruitful. For more details about classical tilting
modules we refer to \cite{AsSS,AnHK}.

Recently Adachi, Iyama and Reiten \cite{AIR} introduced
$\tau$-tilting
 theory to generalize the classical tilting theory from viewpoint of
 mutations. This is very close to the silting theory
 (e.g.\ \cite{AiI,DF,HKM,KV}) and the cluster tilting
 theory (e.g.\ \cite{BMRRT,IY,KR}).
 The central notion of $\tau$-tilting theory is support $\tau$-tilting modules, and therefore it is important to classify
 support $\tau$-tilting modules for a given algebra.
Recently some authors worked on this topic, e.g. Adachi \cite{A1}
classified $\tau$-rigid modules for Nakayama algebras, Adachi
\cite{A2} and Zhang \cite{Z1} studied $\tau$-rigid modules for
algebras with radical square zero, and Mizuno \cite{M} classified
support $\tau$-tilting modules for preprojective algebras of Dynkin
type. In this context, it is basic to consider algebras with only
finitely many support $\tau$-tilting modules, called $\tau$-tilting
finite algebras and studied by Demonet, Iyama and Jasso \cite{DIJ}.
For more details of $\tau$-tilting theory, we refer to \cite{AAC,
AIR, AnMV, DIRRT, HuZ, J, IJY, IRRT,W, Zh} and so on.

In this paper we focus on classifying tilting modules and support
$\tau$-tilting modules over a class of Auslander algebras. Recall
that an algebra $\Lambda$ is called an Auslander algebra if the
global dimension of $\Lambda$ is less than or equal to 2 and the
dominant dimension of $\Lambda$ is greater than or equal to 2. It is
showed by Auslander there is a one-to-one correspondence between
Auslander algebras and algebras of finite representation type.

In the rest, let $\Lambda$ be the Auslander algebra of the algebra
$K[x]/(x^n)$. Then $\Lambda$ is presented by the quiver
\[\xymatrix{
1\ar@<2pt>[r]^{a_1}&2\ar@<2pt>[r]^{a_2}\ar@<2pt>[l]^{b_2}&3\ar@<2pt>[r]^{a_3}\ar@<2pt>[l]^{b_3}&\cdots\ar@<2pt>[r]^{a_{n-2}}\ar@<2pt>[l]^{b_4}&n-1\ar@<2pt>[r]^{a_{n-1}}\ar@<2pt>[l]^{b_{n-1}}&n\ar@<2pt>[l]^{b_n}
}\] with relations $a_{1} b_{2}= 0$ and $a_{i} b_{i+1} =b_{i}
a_{i-1}$ for any $2 \leq i \leq n-1$. All modules in this paper are
right modules. Denote by $\tilt\Lambda$ the set of isomorphism
classes of basic tilting $\Lambda$-modules. We show that each
tilting $\Lambda$-module is isomorphic to a product of maximal
ideals $I_{1}, \ldots , I_{n-1}$ of $\Lambda$. Moreover, we show a
strong relationship between basic tilting $\Lambda$-modules and the
symmetric group $\mathfrak{S}_{n}$.

For $w,w' \in\mathfrak{S}_{n}$ and $1\leq i\leq n$, we denote the product $w'w \in \mathfrak{S}_{n}$ by
$(w'w)(i) := {w'}(w(i))$.
 Denote by $s_{i} \in\mathfrak{S}_{n}$ the transposition $(i,i+1)$  for $1 \leq i \leq n-1$.
 The {\it length} of $w \in \mathfrak{S}_{n}$ is defined by
$ l (w) := \# \{ (i,j) \mid 1 \leq i < j \leq n , w(i) > w(j) \}$
and an expression $ w = s_{i_{1}}s_{i_{2}} \ldots s_{i_{l}} $ of $w
\in \mathfrak{S}_{n}$ is called a {\it reduced expression} if $l =
l(w)$. For elements $w,w' \in \mathfrak{S}_{n}$, if $l(w') = l(w) +
l(w' w^{-1})$ then we write $w \leq w'$. This gives a partial order
on $\mathfrak{S}_{n}$ called the {\it left order}. The Hasse quiver
of $\mathfrak{S}_{n}$ has vertices $w$ corresponding to each element
$w\in\mathfrak{S}_{n}$, and has arrows $w \rightarrow s_{i}w$ if $
l(w) > l(s_{i}w)$ and $w\leftarrow s_{i}w$ if $ l(w) < l(s_{i}w)$
for $w \in \mathfrak{S}_{n}$ and $1 \leq i \leq n-1$. Now we are in
a position to state our first main result.

\begin{theorem}[Theorems \ref{3.10}, \ref{3.17}]\label{1.1}
Let $\Lambda$ be the Auslander algebra of $K[x]/(x^n)$, and $\langle
I_{1} , \ldots , I_{n-1} \rangle$ the ideal semigroup of $\Lambda$
generated by the maximal ideals $I_{1},\ldots,I_{n-1}$.
\begin{enumerate}[\rm(1)]
\setlength{\itemsep}{0pt}
\item The set $\tilt\Lambda$ is given by $\langle I_{1} , \ldots , I_{n-1} \rangle$.
\item There exists a well-defined bijection $I:\mathfrak{S}_{n}
      \cong\langle I_{1}, \ldots,I_{n-1} \rangle$, which maps $w$ to $I(w) = I_{i_{1}} \cdots
      I_{i_{l}}$ where $w = s_{i_{1}} \cdots s_{i_{l}}$ is an arbitrary
      reduced expression.
\item Consequently there exists a bijection $I:\mathfrak{S}_n\cong\tilt\Lambda$. In particular $\#\tilt\Lambda=n!$.
\item The map $I$ in (3) is an anti-isomorphism of posets.
\end{enumerate}
\end{theorem}

 Theorem \ref{1.1}(3) has been shown in \cite{BHRR} by using a combinatorial method.
Our method in this paper is rather homological, and we shall modify
the method in \cite{IR,BIRS,M} for preprojective algebras to the
Auslander algebra of $K[x] / (x^{n})$ by using basic properties of
Auslander algebras in Section $2$.

 Denote by $\sttilt\Lambda$ the set of isomorphism classes of basic support $\tau$-tilting $\Lambda$-modules,
  and by $\mu_i(T)$ the mutation of $T$ with respect to the $i$-th indecomposable direct summand of
  $T$. The set $\sttilt\Lambda$ forms a poset (=partially ordered set) with
respect to the generation order (Definition \ref{2.r}). We show the following main result
of this paper in Section 4, where the map
$I:\mathfrak{S}_{n+1}\cong\sttilt\Lambda$ is an extension of the map
$I$ in Theorem 1.1.

\begin{theorem}[Theorems \ref{4.14}, \ref{4.13}, \ref{4.16}]\label{1.2}
Let $\Lambda$ be the Auslander algebra of $K[x]/(x^n)$.
\begin{enumerate}[\rm(1)]
\item $\sttilt\Lambda$ is a disjoint union of $\mu_{i+1}\mu_{i+2}\cdots\mu_n(\tilt\Lambda)$ for $0\le i\le n$.

\item There exists a bijection $I:\mathfrak{S}_{n+1}\cong\sttilt\Lambda$
which maps $w$ to $I(w)=\mu_{i_1}\mu_{i_2}\cdots\mu_{i_l}(\Lambda)$,
where $w=s_{i_{1}}s_{i_{2}}\cdots s_{i_{l}}$ is an arbitrary
expression. In particular, we have $\#\sttilt\Lambda=(n+1)!$.
\item The map $I$ in (2) is an anti-isomorphism of posets.
\end{enumerate}
\end{theorem}

Now let $\Gamma$ be the preprojective algebra of Dynkin type $A_n$.
Then there exists a natural surjection $\Lambda\to\Gamma$, and we
get a tensor functor
$-\otimes_{\Lambda}\Gamma:\mod\Lambda\rightarrow\mod\Gamma$. By
using this we get a bijection between $\sttilt\Lambda$ and
$\sttilt\Gamma$. More precisely, we have:

\begin{theorem}[Theorem \ref{5.4}]\label{1.4}
Let $\Lambda$ and $\Gamma$ be as above. Then
\begin{enumerate}[\rm(1)]
\item The map $-\otimes_{\Lambda}\Gamma:\sttilt\Lambda\rightarrow\sttilt\Gamma$ given by $U\mapsto
U\otimes_{\Lambda}\Gamma$ is bijective.
\item The map in $(1)$ is an isomorphism of posets.
\end{enumerate}
\end{theorem}

As a corollary of Theorems \ref{1.2} and \ref{1.4}, we recover
Mizuno's anti-isomorphism $\mathfrak{S}_{n+1}\to\sttilt\Gamma$
\cite[Theorems 2.21 and 2.30]{M} since it is the composition of
$-\otimes_{\Lambda}\Gamma$ in Theorem \ref{1.4} and $I$ in Theorem
\ref{1.2}.

\begin{corollary}[Corollary \ref{5.6}]\label{1.5}
Let $\Lambda$ and $\Gamma$ be as above. There are isomorphisms
between the following posets:
\begin{enumerate}[\rm(1)]
\item The poset $\sttilt\Lambda$ with the generation order.
\item The poset $\sttilt\Gamma$ with the generation order.
\item The symmetric group $\mathfrak{S}_{n+1}$ with the opposite of the left order.
\item The poset $\sttilt(\Lambda^{\op})$ with the opposite of the generation order.
\item The poset $\sttilt(\Gamma^{\op})$ with the opposite of the generation order.
\item The symmetric group $\mathfrak{S}_{n+1}$ with the right order.
\end{enumerate}
\end{corollary}
 The paper is organized as follows:
In Section 2, we recall some preliminaries on Auslander algebras,
tilting modules and support $\tau$-tilting modules. In
Section 3, we focus on the tilting modules over the Auslander
algebra of $K[x]/(x^n)$ and we prove Theorem \ref{1.1}. In
Section 4, we use Theorem \ref{1.1} and some other facts of tilting
modules to prove Theorem \ref{1.2}. Finally, in Section 5, we apply
Theorem \ref{1.2} and Theorem \ref{1.4} to preprojective algebras of
Dynkin type $A_n$ and get Mizuno's bijection for preprojective
algebras of Dynkin type $A_n$.

Throughout this paper, we denote by $K$ an arbitrary field, and we
consider basic finite dimensional $K$-algebras.
By a module, we mean a finitely generated right module.
For an algebra $A$, we denote by
 $\mod A$ the category of finitely generated right $A$-modules.
For an $A$-module $M$, we denote by $\add M$ the full
subcategory of $\mod A$ whose objects are direct summands of
$M^{n}$ for some $n > 0$. The composition of homomorphisms $f:X
\rightarrow Y$ and $g:Y \rightarrow Z$ is denoted by $gf:X
\rightarrow Z$.
Thus $\Hom_\Lambda(X,Y)$ is an $\End_\Lambda(Y)^{\rm op}$-module and
an $\End_\Lambda(X)$-module.

For more recent results on $\tau$-tilting theory of Auslander algebras, we refer to \cite{IZ,Z2}.

\bigskip
\noindent{\bf Acknowledgement} Theorem \ref{1.1} was obtained in the
Master thesis of Yusuke Tsujioka \cite{T}, who was a student of the
first author in Graduate school of Mathematics in Nagoya University.
The authors thank him for allowing them to include his results in
this paper. Other parts of this paper were done when the second
author visited Nagoya University in the year 2015. The second author
would like to thank Laurent Demonet, Takahide Adachi, Yuta Kimura,
Yuya Mizuno and Yingying Zhang for useful discussion and kind help.
He also wants to thank the first author for hospitality during his
stay in Nagoya. Both of the authors would like to thank the referees
for useful suggestions to improve this paper.

\section{Preliminaries}

In this section we recall some basic properties of Auslander
algebras, tilting modules and support $\tau$-tilting modules. We
begin with the definition of Auslander algebras.

For an algebra $\Lambda$ and a $\Lambda$-module $M$, denote by $\gldim\Lambda$ the global dimension of $\Lambda$, and by
$\projdim M$ (resp. $\injdim M$) the projective dimension
(resp. injective dimension) of $M$. We recall the following definition.

\begin{definition}
An algebra $\Lambda$ is called an {\it Auslander algebra}  if $\gldim  \Lambda\leq2$
and $E_i(\Lambda)$ is projective for $i=0,1$, where $E_i(\Lambda)$ is the $(i+1)$-th term in a minimal injective
resolution of $\Lambda$.
\end{definition}

Recall that an algebra $R$ is called {\it representation-finite} if
$\mod R$ admits an additive generator $M$, that is, $\mod R=\add M$.
The following classical result in \cite{AuRS} shows the relationship
between representation-finite algebras and Auslander algebras.

\begin{theorem}\label{2.2}
\begin{enumerate}[\rm(1)]
\item For an additive generator $M$ of the category
 $\mod R$ over a representation-finite algebra $R$, the algebra
 $\End_{R} (M )$ is an Auslander algebra.
 \item For an Auslander algebra $\Lambda$ and an
       additive generator $Q$ of the category of
       projective-injective $\Lambda$-module, the algebra
       $\End_{\Lambda} (Q)$ is
       representation-finite.
\item The correspondences in $(1)$ and $(2)$ induce mutually inverse bijections between Morita
       equivalence classes of representation-finite algebras and
       Morita equivalence classes of Auslander algebras.
\end{enumerate}
\end{theorem}

We call $\Lambda=\End_{R} (M)$ in Theorem \ref{2.2}(1) an
{\it Auslander algebra} of $R$. In this case, for $X \in \mod R$ we denote
\[P_{X} = \Hom_{R} (M,X),\ P^{X} = \Hom_{R} (X,M),\
S_{X} =P_{X} / \rad P_{X}\ \mbox{ and }\ S^{X} = P^{X} /\rad
P^{X}.\] Here $P_-=\Hom_R(M,-)$ is an equivalence between $\add M$ and
$\add\Lambda$, and $P^-=\Hom_R(-,M)$ is a duality between $\add M$ and
$\add\Lambda^{\rm op}$. The following statement \cite{AuRS} shows the
relationship between almost split sequences of $R$ and projective
resolutions of simple $\Lambda$-modules.

\begin{proposition}\label{2.3}
Let $\Lambda$ be an Auslander algebra of $R$ and let $X$ be an indecomposable $R$-module. Then we have
\begin{enumerate}[\rm(1)]
\item $\projdim(S_X)_{\Lambda} \leq 1$ if and only if $X$
 is a projective $R$-module. Then
$0 \rightarrow P_{\rad X} \rightarrow P_{X} \rightarrow S_{X}
 \rightarrow 0$
 is a minimal projective resolution of $S_{X}$.
\item $\projdim(S_{X})_{\Lambda} = 2$ if and only if
       $X$ is a nonprojective $R$-module. Then the almost split sequence
 $0 \rightarrow \tau X \rightarrow E \rightarrow X \rightarrow 0$ gives
      a minimal projective resolution $0 \rightarrow P_{\tau X} \rightarrow
 P_{E} \rightarrow P_{X} \rightarrow S_{X} \rightarrow 0$ of $S_{X}$.
\item $\projdim_{\Lambda}(S^{X}) \leq 1$ if and only if $X$
 is an injective $R$-module. Then
 $0 \rightarrow P^{X / \mathrm{soc}X} \rightarrow P^{X} \rightarrow S^{X}
 \rightarrow 0$ is a minimal projective resolution of $S^{X}$.
\item $\projdim_{\Lambda}(S^{X}) = 2$ if and only if
       $X$ is a noninjective $R$-module. Then the almost split sequence
 $0 \rightarrow X \rightarrow E \rightarrow \tau^{-1} X \rightarrow 0$
      gives a minimal projective resolution
 $0 \rightarrow P^{\tau^{-1} X} \rightarrow
 P^{E} \rightarrow P^{X} \rightarrow S^{X} \rightarrow 0$ of
 $S^{X}$.
 \end{enumerate}
\end{proposition}

 Denote by $(- )^{\ast}=\rm
\Hom_{\Lambda}(-,\Lambda)$. We also need the following lemma.

\begin{lemma}\label{2.4}
Let $\Lambda$ be an Auslander algebra of $R$ and let $X$ be an indecomposable
nonprojective $R$-module. Then we have
\begin{enumerate}[\rm(1)]
\item $\Ext_{\Lambda}^{2} ( S_{X}, \Lambda )\cong S^{\tau X}$, and
$\Ext_{\Lambda}^{i} ( S_{X}, \Lambda )=0$ if $i\neq  2$.
\item $\Ext_{\Lambda}^{i}(S_{X} , Y) \cong \Tor_{2-i}^{\Lambda}(Y,S^{\tau X})$ for
$Y \in\mod \Lambda$ and $i\in\mathbb{Z}$.
\end{enumerate}
\end{lemma}

\begin{proof}
We only prove (2) since the statement (1) follows from (2) immediately.

By Proposition \ref{2.3}, there exist projective resolutions
\begin{eqnarray}\label{2.4.1}
&0 \rightarrow P_{\tau X} \rightarrow P_{E} \rightarrow P_{X}
 \rightarrow S_{X} \rightarrow 0,&\\ \label{2.4.2}
&0 \rightarrow P^{X} \rightarrow P^{E} \rightarrow P^{\tau X}
 \rightarrow S^{\tau X} \rightarrow 0.&
 \end{eqnarray}
of $S_{X}$ and $S^{\tau X}$, respectively. Applying
$\Hom_{\Lambda} (-,Y)$ to \eqref{2.4.1}, we obtain a complex
\begin{equation}\label{2.4.3}
0 \rightarrow \Hom_{\Lambda} (P_{X}, Y)
\rightarrow \Hom_{\Lambda} (P_{E}, Y)
\rightarrow \Hom_{\Lambda} (P_{\tau X}, Y )
\rightarrow 0
\end{equation}
whose homologies are $\Ext_{\Lambda}^{i} (S_{X},Y)$. Similarly, applying
$Y \otimes_{\Lambda} -$ to \eqref{2.4.2}, we obtain a complex
\begin{equation}\label{2.4.4}
0 \rightarrow Y \otimes_{\Lambda} P^{X} \rightarrow Y
\otimes_{\Lambda} P^{E} \rightarrow Y \otimes_{\Lambda} P^{\tau X}
\rightarrow 0
\end{equation}
whose homologies are $\Tor_{2-i}^{\Lambda}(Y,S^{\tau X})$. Because
$\Hom_{\Lambda} (P_{-},Y) \cong Y \otimes_{\Lambda}
{P_{-}}^{\ast}\cong Y \otimes_{\Lambda} P^{-}$ holds, \eqref{2.4.3}
and \eqref{2.4.4} are isomorphic. Thus we obtain the desired
isomorphism.
\end{proof}

The following lemma is useful.

\begin{lemma}\label{2.5} Let $\Lambda$ be an Auslander algebra
and $Y\in \mod \Lambda$. Then any composition factor of $\Ext_{\Lambda}^{2} (Y, \Lambda)$ has
projective dimension $2$.
\end{lemma}

\begin{proof}
Without loss of generality, we can assume that $Y$ is simple since any composition factor of $\Ext^2_\Lambda(Y,\Lambda)$ is a composition factor of $\Ext^2_\Lambda(S,\Lambda)$ for some simple $\Lambda$-module $S$.
If $\projdim Y\le 1$, then the assertion is clear since the zero module has no composition factor.
If $\projdim Y=2$, then Proposition \ref{2.3}(2) shows that $Y=S_X$ for some indecomposable nonprojective $R$-module $X$.
Thus $\Ext^2_\Lambda(Y,\Lambda)=S^{\tau X}$ holds by Lemma \ref{2.4}(2), and the assertion follows from Proposition \ref{2.3}(4).
%We prove the assertion by induction on the length of $X$, which is
%denoted by $l(X)$.
%If $l(X)=1$, then $\Ext_{\Lambda}^{2} (X, \Lambda)$ is simple by
%Proposition \ref{2.4}(1). By Proposition \ref{2.3}(4), the
%projective dimension is $2$. Assume that it is true for $l(X)<t$.
%For the case $l(X)=t$, take an exact sequence $0\rightarrow
%X'\rightarrow X\rightarrow X''\rightarrow 0$ such that $l(X')<t$ and
%$l(X'')<t$ hold. Applying $(-)^{\ast}$, one gets an exact sequence
%$\Ext_{\Lambda}^{2} (X'', \Lambda)\rightarrow\Ext_{\Lambda}^{2} (X,
%\Lambda)\rightarrow\Ext_{\Lambda}^{2} (X', \Lambda)$. Since any
%composition factor of $\Ext_{\Lambda}^{2} (X, \Lambda)$ is either
%the composition factor of $\Ext_{\Lambda}^{2} (X', \Lambda)$ or that
%of $\Ext_{\Lambda}^{2} (X'', \Lambda)$, we are done.}
\end{proof}

We also need the following general result on algebras of global
dimension 2.

\begin{lemma}\label{2.6}
Let $\Lambda$ be an algebra with $\gldim  \Lambda \leq 2$ and $Y \in
\mod \Lambda$. Then $Y^{\ast \ast}$ is a projective
$\Lambda$-module.
\end{lemma}

\begin{proof}
Let $Q_{1} \rightarrow Q_{0} \rightarrow Y  \rightarrow 0$ be a
projective presentation of $Y$. Applying $(- )^{\ast}$, we obtain an
exact sequence
 $ 0 \rightarrow Y^{\ast} \rightarrow Q_{0}^{\ast}
  \rightarrow Q_{1}^{\ast}$.
Hence $Y^{\ast}$ is a projective $\Lambda^{\op}$-module, since
 $Q_{0}^{\ast}$ and $Q_{1}^{\ast}$ are projective $\Lambda^{\op}$-modules
 and $\gldim  \Lambda \leq 2$.
Thus $Y^{\ast \ast}$ is a projective $\Lambda$-module.
\end{proof}

By the lemma above we obtain the following.

\begin{lemma}\label{2.7}
Let $\Lambda$ be an Auslander algebra, and let $Y$ be a
$\Lambda$-module with $\projdim Y\le 1$. Then the evaluation map
$\varphi_Y:Y \to Y^{\ast \ast}$ is injective, and the projective
dimension of any composition factor of $Y^{\ast \ast} / Y$ is $2$.
\end{lemma}

\begin{proof} By \cite{AuB}, we get an exact sequence
$0\rightarrow\Ext_{\Lambda^{\op}}^{1} ( \Tr Y , \Lambda)\rightarrow
Y\rightarrow Y^{\ast\ast}\rightarrow\Ext_{\Lambda^{\op}}^{2} (\Tr Y,
\Lambda)\rightarrow 0$. Then the latter assertion holds by Lemma
\ref{2.5}. We prove the former one in two steps.

(1) We show that the projective dimension of any
composition factor of $\Tr Y$ is $2$.

It suffices to show that
$\Hom_{\Lambda^{\op}}(P,\Tr Y)=0$ holds for the projective cover $P$
of any simple $\Lambda^{ \op}$-module $S$ with $\projdim S\leq1$. By
Proposition \ref{2.3}(3), $P=P^I$ for some injective $R$-module $I$.
On one hand, take a minimal projective resolution of $Y$:
\begin{equation}\label{2.7.1}
0\rightarrow P_{X_1}\stackrel{P_f}{\rightarrow} P_{X_0}\rightarrow
Y\rightarrow0
\end{equation}
Since $M$ is a generator, then we get an $R$-module monomorphism
$f:X_1\rightarrow X_0$. Applying $\Hom_{R}(-,I)$, one has an
epimorphism
\begin{equation}\label{(2.6)}
\Hom_{R}(X_0,I)\rightarrow \Hom_{R}(X_1,I).
\end{equation}
On the other hand, applying the functor $(-)^{\ast}$ to \eqref{2.7.1}, we
get an exact sequence $P^{X_0}\rightarrow P^{X_1}\rightarrow \Tr
Y\rightarrow 0$. Then applying the functor
$\Hom_{\Lambda^{\op}}(P^I,-)$, one obtains an exact sequence
\begin{equation*}%\label{2.7.3}
\Hom_{\Lambda^{\op}}(P^I,P^{X_0})\rightarrow
\Hom_{\Lambda^{\op}}(P^I,P^{X_1})\rightarrow
\Hom_{\Lambda^{\op}}(P^I,\Tr Y)\rightarrow 0
\end{equation*}
This can be rewritten as $\Hom_{R}(X_0,I)\rightarrow
\Hom_{R}(X_1,I)\rightarrow \Hom_{\Lambda^{\rm op}}(P^I,\Tr
Y)\rightarrow 0$. Thus we obtain $\Hom_{\Lambda^{\rm op}}(P^I,\Tr Y)=0$ by
\eqref{(2.6)}.

(2) Now we prove the assertion.
By (1) and Proposition \ref{2.3}(4), any composition factor of $\Tr Y$ has the form $S^X$ for some indecomposable noninjective $R$-module $X$.
By the dual of Lemma \ref{2.4}(1), we have $\Ext^1_{\Lambda^{\rm op}}(S^X,\Lambda)=0$. Thus $\Ext^1_{\Lambda^{\rm op}}(\Tr Y,\Lambda)=0$.
%We show $\Ext_{\Lambda^{\op}}^{1} (\Tr X, \Lambda)=0$ by using (1).
%For any simple $\Lambda^{\rm op}$-module $S$, we have $\Ext_{\Lambda^{\rm op}}^1(S,\Lambda)\cong \Hom_{\Lambda^{\rm op}}(S,E^1(\Lambda))$.
%Let $S$ be a composition factor of  $\Tr X$. Since the projective dimension of $S$ is $2$ by (1) and $\Lambda$ is an Auslander algebra,
%then one gets that $\Hom_{\Lambda^{\rm op}}(S,E^1(\Lambda))=0$, and hence $\Ext_{\Lambda^{\rm op}}^1(S,\Lambda)=0$. By using induction on the length of $\Tr X$, one gets the assertion.
\end{proof}

In the rest of this section, $\Lambda$ is an arbitrary algebra. In
the following we recall some basic properties of tilting modules. We
begin with the definition of tilting modules.
\begin{definition}\label{2.8}
We call $T \in \mod \Lambda$ a {\it tilting module} if $T$ satisfies
the following conditions
\begin{enumerate}
\item[(T1)] $\projdim T \leq 1$.
\item[(T2)] $\Ext_{\Lambda}^{1}(T,T) = 0$.
\item[(T3)] There exists a short exact sequence
$0\rightarrow \Lambda \rightarrow T_{0} \rightarrow T_{1}
  \rightarrow 0$ with $T_{0}, T_{1} \in \add  T$.
\end{enumerate}
The condition (T3) is equivalent to
\begin{enumerate}
\item[(T3')] The number of non-isomorphic direct summands of
$T$ is equal to that of $\Lambda$.
\end{enumerate}
\end{definition}

Now let us recall some general properties of tilting modules
\cite{HaU}.

\begin{lemma}\label{2.9}
Let $T$ be a tilting $\Lambda$-module, and let
 $0 \rightarrow Q_{1} \rightarrow Q_{0} \rightarrow T
 \rightarrow 0$ be a minimal projective resolution of $T$. Then we have the
following:
\begin{enumerate}[\rm(1)]
\item $(\add  Q_{1}) \cap (\add  Q_{0}) = 0$ and
 $\add  (Q_{0} \oplus Q_{1}) = \add \Lambda$ hold.
\item For a simple $\Lambda$-module $S$, precisely one of
       $\Hom_{\Lambda} (T,S) = 0$ and
 $\Ext_{\Lambda}^{1} (T, S) = 0$ holds.
\item For a simple $\Lambda^{\op}$-module $S$, precisely one of $T
      \otimes_{\Lambda} S = 0$ and $\Tor^{\Lambda}_{1} (T,
      S) = 0$ holds.
\end{enumerate}
\end{lemma}

We also have the following properties for the tensor products of
tilting modules.

\begin{proposition}\label{2.10}
Let $T$ be a tilting $\Lambda$-module
with $\Gamma = \End_{\Lambda} (T)$.
\begin{enumerate}[\rm(1)]
\item Let $U$ be a tilting $\Gamma$-module. If
      $\Tor_{i}^{\Gamma} (U,T) = 0$ for any $i > 0$
      and $\projdim  (U \otimes_{\Gamma} T) \leq
      1$, then $U \otimes_{\Gamma} T$ is a tilting $\Lambda$-module with
      $\End_{\Lambda} (U \otimes_{\Gamma} T) \cong
      \End_{\Gamma} (U)$.
\item Let $V$ be a tilting $\Lambda$-module. If
      $\Ext_{\Lambda}^{i} (T,V) = 0$ for any $i >
      0$ and $\projdim \Hom_{\Lambda} (T,V
     )_{\Gamma} \leq 1$, then $\Hom_{\Lambda} (T,V
     )$ is a tilting $\Gamma$-module with $\End_{\Gamma}
      (\Hom_{\Lambda} (T,V)) \cong
      \End_{\Lambda} (V)$.
\end{enumerate}
\end{proposition}

\begin{proof}
(1) Since $-\otimes_{\Gamma}^{\mathbf{L}}T:\DD^{\bo}(\mod\Gamma)\to\DD^{\bo}(\mod\Lambda)$ is a triangle equivalence, $U \otimes_{\Gamma}^{\mathbf{L}} T$ is a tilting complex of $\Lambda$.
Since $\Tor_{i}^{\Gamma} (U,T) = 0$ for any $i > 0$ by our assumption, $U \otimes_{\Gamma} T \cong U
\otimes_{\Gamma}^{\mathbf{L}} T$ holds.
Since $\projdim  (U \otimes_{\Gamma} T) \leq 1$, the assertion holds.
One can show (2) similarly.
\end{proof}

Denote by $\tau$ the AR-translation and denote by $|N|$ the number
of non-isomorphic indecomposable direct summands of $N$ for a
$\Lambda$-module $N$. In the following we recall some basic
properties of $\tau$-tilting theory. Firstly, we need the following
definition in \cite{AIR}.

\begin{definition}\label{2.17}
\begin{enumerate}[\rm(1)]
\item We call $N \in \mod \Lambda$ {\it $\tau$-rigid} if ${\rm
Hom}_{\Lambda}(N, \tau N) = 0$.
\item We call $N \in \mod \Lambda$ {\it $\tau$-tilting} if $N$ is
$\tau$-rigid and $|N| = |\Lambda|$.
\item We call $N \in \mod \Lambda$ {\it support $\tau$-tilting} if there
exists a basic idempotent $e$ of $\Lambda$ such that $N$ is a
$\tau$-tilting $(\Lambda/(e))$-module.
In this case, we call $(N,e\Lambda)$ a \emph{support $\tau$-tilting pair}.
\end{enumerate}
\end{definition}

It is clear that every tilting $\Lambda$-module is a $\tau$-tilting
$\Lambda$-module, and hence a support $\tau$-tilting module.
Moreover, it is showed in \cite{AIR} tilting $\Lambda$-modules are
exactly faithful support $\tau$-tilting modules.
Clearly any support $\tau$-tilting pair $(N,e\Lambda)$ satisfies
$|N|+|e\Lambda|=|\Lambda|$.

For a torsion class $\TT$ in $\mod\Lambda$, we denote by $P(\TT)$ the direct sum of one copy of each of the indecomposable Ext-projective objects in $\TT$ up to isomorphism.
The following properties of $\tau$-rigid modules are important.

\begin{definition-proposition}{\cite[Theorem 2.10]{AIR}}\label{2.18}
Let $\Lambda$ be an algebra and let $U$ be a $\tau$-rigid module.
Then $T=P({}^\perp\tau U)$ is a $\tau$-tilting $\Lambda$-module,
where ${}^\perp\tau U$ consists of $\Lambda$-modules $X$ satisfying
$\Hom_{\Lambda}(X,\tau U)=0$.
We call $T$ the \emph{Bongartz completion} of $U$.
\end{definition-proposition}

Recall that $\sttilt\Lambda$ is the set of isomorphism classes of basic support $\tau$-tilting $\Lambda$-modules.
For a $\Lambda$-module $X$, we define a full subcategory of $\mod\Lambda$ by
\begin{eqnarray*}
\Fac X= \{ Y\in \mod \Lambda \mid \text{There exists an epimorphism
}X^{n} \rightarrow Y \text{ for some } n \geq 0 \}&
\end{eqnarray*}
Now we define the partial order on $\sttilt\Lambda$ as follows:

\begin{definition}\label{2.r} For
basic support $\tau$-tilting $\Lambda$-modules $T,U$, we write
$T \leq U$ if $\Fac T \subseteq \Fac U$. Then the relation
$\leq$ gives a partial order on the set $\sttilt\Lambda$ by \cite[Theorem 2.7]{AIR}. We call this partial order a {\it generation order}.
\end{definition}
Clearly $\Lambda$ is a unique maximal element and $0$ is a unique
minimal element in $\sttilt\Lambda$.

We now recall the Hasse quiver of general posets.

\begin{definition}\label{2.14}
The {\it Hasse quiver} ${\rm H}(P)$ of a poset $(P, \le)$ is defined
as follows:
 \begin{enumerate}[\rm(1)]

\item The vertices are the elements of the poset
$P$.

\item For $X, Y \in P$, there is an arrow $X \rightarrow Y$
if and only if $X > Y$ and there is no $Z \in P$ satisfying $X > Z > Y$.
\end{enumerate}
\end{definition}

The following observation is clear.

\begin{lemma}\label{2.15}
Two partial orders on a finite set are the same if and only if their Hasse quivers are the same.
\end{lemma}
Now it is time to recall the mutations of support $\tau$-tilting
modules from \cite{AIR}.

\begin{definition}\label{mutation}
Let $T,T'\in\sttilt\Lambda$.
We call $T'$ a {\it mutation of $T$} if $T$ and $T'$ have the same indecomposable direct summands except one. Precisely speaking, one of the following three cases occurs, where $(T,P)$ and $(T',P')$ are the support $\tau$-tilting pairs.
\begin{enumerate}[\rm(1)]
\item $T=V\oplus X$ and $T'=V\oplus X'$ with $X\not\cong X'$ indecomposable;
\item $T=T'\oplus X$ and $P'=P\oplus Q'$ with $X$ and $Q'$ indecomposable.
\item $T'=T\oplus X'$ and $P=P'\oplus Q$ with $X'$ and $Q$ indecomposable;
\end{enumerate}
We call $T'$ a mutation of $T$ {\it at $X$} in cases (1)(2),
and {\it at $Q$} in case (3). It is uniquely determined by $T$ and
the indecomposable direct summand $X$ or $Q$ of $T$ or $P$ respectively.

We call $T'$ a {\it left mutation} (resp. {\it
right mutation}) of $T$ if $\Fac T'\subsetneq \Fac T$ (resp. $\Fac
T'\supsetneq \Fac T$).
\end{definition}

In the following we give a method of calculating left mutations of
support $\tau$-tilting modules due to Adachi, Iyama and Reiten
\cite{AIR}.

\begin{theorem}{\cite[Theorem 2.30]{AIR}\cite[Theorem 1.2] {Zh}}\label{2.20}
Let $T = V \oplus X$ be a basic $\tau$-tilting $\Lambda$-module
which is the Bongartz completion of $V$, where $X$ is
indecomposable. Let $X\stackrel{f}{\rightarrow} V'\stackrel{
g}{\rightarrow}Y \rightarrow 0$ be an exact sequence, where $f$ is a
minimal left $(\add V)$-approximation.
Then $Y$ is either indecomposable or zero, and $V\oplus Y$ is
a left mutation of $T$ at $X$ in both cases.
\end{theorem}

Now let us recall the relationship between mutations and the Hasse
quiver, which is given in \cite{HaU, RS} for $\tilt\Lambda$ and in
\cite{AIR} for $\sttilt\Lambda$.

\begin{theorem}\label{2.m}
Let $T,U\in\sttilt\Lambda\ (resp.\ \tilt\Lambda)$. The following are
equivalent.
\begin{enumerate}[\rm(1)]
\item $T$ is a left mutation of $U$.
\item $U$ is a right mutation of $T$.
\item $U>T$ and there is no $V\in \sttilt\Lambda \ (resp.\  \tilt\Lambda) $ such that
$U>V>T$.
\item There is an arrow from $U$ to $T$ in ${\rm H}(\sttilt\Lambda)\ (resp.\  {\rm H}(\tilt\Lambda))$.
\end{enumerate}
\end{theorem}

The following result \cite[Corollary 2.38]{AIR} gives a method of
judging an algebra to be $\tau$-tilting finite.

\begin{proposition}\label{2.22}
If ${\rm H}(\sttilt\Lambda)$ admits a finite connected component
${\rm C}$, then ${\rm H}(\sttilt\Lambda)={\rm C}$.
\end{proposition}

\section{Tilting modules over the Auslander algebra of $K[x]/(x^n)$}

Throughout this section, let $R=K[x]/(x^n)$ be a factor algebra of the polynomial ring $K[x]$ with $n\ge1$, and $\Lambda$ the Auslander algebra of $R$.
Then the AR-quiver of $R$ is
\[\xymatrix{
K\ar@<2pt>[r]&K[x]/(x^2)\ar@<2pt>[r]\ar@<2pt>[l]&K[x]/(x^3)\ar@<2pt>[r]\ar@<2pt>[l]&\cdots\ar@<2pt>[r]\ar@<2pt>[l]&K[x]/(x^{n-1})\ar@<2pt>[r]\ar@<2pt>[l]&K[x]/(x^n),\ar@<2pt>[l]
}\]
and the Auslander algebra $\Lambda$ is presented by the quiver
\[\xymatrix{
1\ar@<2pt>[r]^{a_1}&2\ar@<2pt>[r]^{a_2}\ar@<2pt>[l]^{b_2}&3\ar@<2pt>[r]^{a_3}\ar@<2pt>[l]^{b_3}&\cdots\ar@<2pt>[r]^{a_{n-2}}\ar@<2pt>[l]^{b_4}&n-1\ar@<2pt>[r]^{a_{n-1}}\ar@<2pt>[l]^{b_{n-1}}&n\ar@<2pt>[l]^{b_n}
}\] with relations $a_{1} b_{2}= 0$ and $a_{i} b_{i+1} =b_{i}
a_{i-1}$ for any $2 \leq i \leq n-1$. In this section, we classify
all tilting $\Lambda$-modules.

 Denote by $\{ e_{1}, \ldots, e_{n}
\}$ a complete set of primitive orthogonal idempotents of $\Lambda$
and denote by $ P_{i} = e_{i} \Lambda$ (resp. $ P^{i} = \Lambda
e_{i}$) the indecomposable projective $\Lambda$-module (resp.
$\Lambda^{\op}$-module). It is easy to see that $P_{1},P_{2}, \ldots ,
P_{n}$ have the following composition series (see $n=4$ for example).\\
\[\left[\begin{smallmatrix}P_1 \end{smallmatrix}\middle| \begin{smallmatrix} P_2\end{smallmatrix}\middle|
\begin{smallmatrix} P_3
\end{smallmatrix}\middle| \begin{smallmatrix}P_4 \end{smallmatrix}\right]
=\left[\begin{smallmatrix} 1\\ &2\\
&&3\\&&&4\\ \end{smallmatrix}\middle| \begin{smallmatrix} &2\\ 1&&3\\
&2&&4\\&&3\\&&&4\end{smallmatrix}\middle| \begin{smallmatrix} &&3\\ &2&&4\\
1&&3\\&2&&4\\&&3\\&&&4\\ \end{smallmatrix}\middle| \begin{smallmatrix}&&&4 \\&&3\\ &2&&4\\
1&&3\\&2&&4\\&&3\\&&&4\\ \end{smallmatrix}\right]
\]

For $1 \leq i \leq n$, we denote by $I_i$ the two-sided ideal generated by $1-e_i$.
This is a maximal left ideal and also a maximal right ideal since there are no loops at the vertex $i$.
Thus we have direct sum decompositions
\[I_{i} = P_{1} \oplus \cdots \oplus (\rad P_{i}) \oplus \cdots \oplus P_{n}
= P^{1} \oplus \cdots \oplus (\rad P^{i}) \oplus \cdots
 \oplus P^{n}.\]
Furthermore, for $1\le i\le n$, we define a $(\Lambda,\Lambda)$-bimodule by $S_{i} = \Lambda / I_{i}$. Clearly we have the
following.

\begin{proposition}\label{3.1}
Let $\Lambda$ be the Auslander algebra of $K[x]/(x^n)$. Then one gets the following.
 \begin{enumerate}[\rm(1)]
\item As a $\Lambda$-module $S_{i} \cong P_{i}
/ \rad P_{i}$ is simple. As a $\Lambda^{\op}$-module $S_{i}
\cong P^{i} / \rad P^{i}$ is simple.

\item There exists an isomorphism $P_{n} \cong
 DP^{n}$ of $\Lambda$-modules. Thus $P_{n}$ is a projective-injective
 $\Lambda$-module.

\item For $1\leq i \leq n-1$, there exist minimal
       projective resolutions of $\Lambda$-modules
\begin{equation*}
0 \rightarrow P_{i} \rightarrow P_{i-1} \oplus P_{i+1} \rightarrow
 P_{i} \rightarrow S_{i} \rightarrow 0\ \mbox{ and }\
  0 \rightarrow P_{i} \rightarrow P_{i-1} \oplus P_{i+1} \rightarrow
 \rad P_{i} \rightarrow 0.
\end{equation*}

\item There exist minimal projective resolutions of
       $\Lambda$-modules
\begin{equation*} 0 \rightarrow P_{n-1} \rightarrow P_{n} \rightarrow S_{n}
 \rightarrow 0\ \mbox{ and }\
 0 \rightarrow P_{n-1} \rightarrow \rad P_{n}
 \rightarrow 0.
\end{equation*}
\end{enumerate}
\end{proposition}

\begin{proof}
(1) is clear. (3) and (4) are immediate from Proposition \ref{2.3} and the AR-quiver of $R$ above.

(2) Since $R$ is a symmetric $K$-algebra, we have an isomorphism $\Hom_R(-,R)\cong D\Hom_R(R,-)$ of functors. This gives the desired isomorphism.
\end{proof}

We need the following properties of tilting $\Lambda$-modules.

\begin{lemma}\label{3.4}
Let $X$ be a $\Lambda$-module. For $1 \leq
i \leq n-1$, there exist isomorphisms
$\Ext^{2}_{\Lambda} (S_{i}, X) \cong
 X \otimes_{\Lambda} S_{i}$ and $\Ext^{1}_{\Lambda}
 (S_{i}, X) \cong \Tor_{\Lambda}^{1} (X,S_{i})$.
If $X$ is tilting, then precisely one of them is zero.
 \end{lemma}

\begin{proof}
%Suppose that $S_i$ was given by $S_Y$ for some $R=K[x]/(x^n)$-module $Y$. Considering the AR-quiver of $R$, then one gets that $Y\cong \tau Y$. Then we get that $S_Y\cong S^{\tau Y}$ by Proposition \ref{3.1}(1).
Since each indecomposable nonprojective $R$-module is $\tau$-stable, we have $\Ext_{\Lambda}^{j}
(S_{i},X) \cong \Tor_{2-j}^{\Lambda} (X,S_{i})$ for $j=1,2$ by Lemma \ref{2.4}(2).
The latter statement follows from Proposition \ref{2.9}(3).
 \end{proof}

Now we are in a position to show the following proposition.

\begin{proposition}\label{3.2}
For $1 \leq i \leq n-1$, $I_{i}$ is a tilting
$\Lambda$-module and a tilting $\Lambda^{\op}$-module.
\end{proposition}

\begin{proof}
We only prove the case of a $\Lambda$-module since the case of a
 $\Lambda^{\op}$-module is similar.
By definition, we have $I_{i} = (
 \bigoplus_{j \neq i} P_{j}) \oplus \rad P_{i}$.

(T1) By Proposition \ref{3.1}(3), we have
 $\mathrm{proj.dim} \rad P_{i} \leq 1$. Thus $\projdim I_{i} \leq 1$.

(T2) It suffices to show that
$\Ext_{\Lambda}^{1}(\rad P_{i} , I_{i}) = 0$.
Since there exists an exact sequence
$0 \rightarrow \rad P_{i}
 \rightarrow P_{i} \rightarrow S_{i} \rightarrow 0$,
we have $\Ext_{\Lambda}^{2} (S_{i} , I_{i}) \cong \Ext_{\Lambda}^{1} (\rad  P_{i} , I_{i})$.
By Lemma \ref{3.4}, we have $\Ext_{\Lambda}^{2}(S_{i} , I_{i}) \cong I_{i} \otimes_{\Lambda}S_{i}$.
On the other hand, we have $ P_{j} \otimes_{\Lambda} S_{i}=e_j\Lambda\otimes_\Lambda S_i=e_jS_i=0$
%\Hom_{\Lambda^{\rm op}}(P^j, S(i))= 0$
for any $j \neq i$. %by Proposition \ref{3.1}(1).
By Proposition \ref{3.1}(3), there
 exists an exact sequence $0 = (P_{i-1} \oplus P_{i+1}) \otimes_{\Lambda} S_{i}
 \rightarrow  (\rad P_{i}) \otimes_{\Lambda} S_{i} \rightarrow 0.$
Thus we have $(\rad P_{i}) \otimes_{\Lambda}S_{i} =0$ and $I_{i} \otimes_{\Lambda} S_{i} = 0$.

(T3) By Proposition \ref{3.1}(3), there exists an exact sequence
$0 \rightarrow \Lambda \rightarrow  (\bigoplus_{j \neq i} P_{j}
) \oplus P_{i-1} \oplus P_{i+1}  \rightarrow  \rad
 P_{i} \rightarrow 0.$
The middle and right terms of this sequence are contained in
 $\add  I_{i}$.
 \end{proof}

Notice that $I_{n}$ is not a tilting $\Lambda$-module. In fact
$I_{n} = (\bigoplus_{i=1}^{n-1}P_{i}) \oplus (\rad P_{n})$ and $\rad
P_{n} \cong P_{n-1}$ hold by Proposition \ref{3.1}(4), and hence
$|I_{n}|=n-1$. This is not possible for tilting $\Lambda$-modules.
%Hence $I_{n}$ can not be a tilting $\Lambda$-module.

To show that any multiplication of ideals $I_{1}, \cdots,I_{n-1}$ is a tilting $\Lambda$-module,
we now prepare the following.

\begin{proposition}\label{3.3}
\begin{enumerate}[\rm(1)]
\item For $1 \leq i \leq n$, we have $\Hom_{\Lambda} (I_{i},S_{i}) = 0$.
\item For $1 \leq i \leq n-1$, the left multiplication $\Lambda\to\End_{\Lambda}(I_{i})$ and
the right multiplication $\Lambda^{\op}\to
\End_{\Lambda^{\op}}(I_{i})$ are isomorphisms.
\end{enumerate}
\end{proposition}

\begin{proof}
(1) For $j \neq i$, we have
$\Hom_{\Lambda} (P_{j},S_{i}) = 0$. Further, by Proposition \ref{3.1}(3)(4),
one gets $\Hom_{\Lambda} (\rad P_{i}, S_{i}) = 0$.
Thus we have $\Hom_{\Lambda} (I_{i},S_{i}) = 0$.

(2) Applying $\Hom_{\Lambda} (-,
\Lambda)$ to a short exact sequence
\begin{equation}\label{3.3.1}
0 \rightarrow I_{i}
\rightarrow \Lambda \rightarrow S_{i} \rightarrow0
\end{equation}
yields a long exact sequence
$0 \rightarrow \Hom_{\Lambda} (S_{i},
 \Lambda) \rightarrow  \Hom_{\Lambda} (\Lambda,
 \Lambda) \rightarrow \Hom_{\Lambda} (I_{i},
 \Lambda) \rightarrow  \Ext_{\Lambda}^{1} (S_{i} , \Lambda) \rightarrow 0$.
Then by Lemma \ref{2.4}, we have $ \Hom_{\Lambda} (S_{i}, \Lambda) =
\Ext_{\Lambda}^{1} (S_{i} , \Lambda) = 0$, and hence
$\Hom_{\Lambda}(I_{i} , \Lambda) \cong \Hom_{\Lambda} (\Lambda, \Lambda) \cong \Lambda$.
On the other hand, applying $\Hom_{\Lambda} (I_{i},-)$ to the short exact sequence \eqref{3.3.1},
one gets an exact sequence
$0 \rightarrow \Hom_{\Lambda} (I_{i} , I_{i}) \rightarrow  \Hom_{\Lambda} (I_{i},
 \Lambda)  \rightarrow \Hom_{\Lambda} (I_{i},S_{i})$.
Using (1), we have $\End_{\Lambda} (I_{i}) \cong \Hom_{\Lambda} (I_{i}, \Lambda) \cong
 \Lambda$.
\end{proof}

From the argument above, we have the following proposition on
 the multiplication of tilting $\Lambda$-modules.

\begin{proposition}\label{3.6}
Let $T$ be a tilting $\Lambda$-module and $1 \leq i \leq n-1$. Then
we have the following.
\begin{enumerate}[\rm(1)]

\item If $TI_{i} \neq T$, then $TI_{i}\cong T \otimes_{\Lambda} I_{i}=T\LotimesL I_i$.

\item $TI_{i}$ is a tilting $\Lambda$-module, and
       $\End_{\Lambda} (TI_{i})\cong
       \End_{\Lambda} (T)$.

\end{enumerate}
\end{proposition}

\begin{proof}

 (1) Since $T I_{i} \neq T$, then $T \otimes_{\Lambda} S_{i} \cong T /
TI_{i} \neq0$, and we have $\Tor_{1}^{\Lambda}
 (T,S_{i}) = 0$ by Proposition \ref{2.9}(3).
Applying $T \otimes_{\Lambda} -$ to the short exact sequence $0
\rightarrow I_{i} \rightarrow \Lambda \rightarrow S_{i} \rightarrow
0$, one gets an exact sequence $0=\Tor_{1}^{\Lambda} (T,
S_{i})\rightarrow T \otimes_{\Lambda} I_{i} \rightarrow
T\otimes_{\Lambda} \Lambda\cong T$. Thus the natural map $T
\otimes_{\Lambda} I_{i} \rightarrow T$ is injective and has the
image $TI_i$. Thus we obtain $T\otimes_{\Lambda} I_{i} \cong
TI_{i}$. Moreover, we have $\Tor_{j}^{\Lambda} (T,I_{i}) \cong
 \Tor_{j+1}^{\Lambda} (T,S_{i})=0$ for $j \geq 1$ since $\projdim T \leq
 1$. Thus $T \otimes_{\Lambda} I_{i}=T\LotimesL I_i$.

(2) If $T I_{i}  = T$, then the assertion is clear. Now assume that
$T I_{i}\not= T$. Since we have $\End_{\Lambda} (I_{i}) \cong \Lambda$
by Proposition \ref{3.3}, $T \otimes_{\Lambda} I_{i} \cong TI_{i}$ is a
tilting module with $\End_{\Lambda} (T)\cong \End_{\Lambda} (TI_{i})$
by (1) and Proposition \ref{2.10}(1).
\end{proof}

 Denote by $ \langle I_{1} , \ldots , I_{n-1} \rangle $
the set of ideals of $\Lambda$ given by products of $I_{1} , \ldots
, I_{n-1}$, where the empty product $\Lambda$ is also contained in
this set. Now we can state the following result.

\begin{theorem}\label{3.7}
Any ideal $T$ in $\langle I_{1} , \ldots , I_{n-1}\rangle$ is a
basic tilting $\Lambda$-module and a basic tilting
$\Lambda^{\op}$-module. The left multiplication
$\Lambda\to\End_{\Lambda} (T)$ and the right multiplication
 $\Lambda^{\op}\to\End_{\Lambda^{\op}} (T)$ are isomorphisms.
 \end{theorem}

\begin{proof}
We only prove the case of a $\Lambda$-module since the case of a $\Lambda^{\op}$-module is similar.

By Proposition \ref{3.2}, each of $I_{1} , \ldots , I_{n-1}$ is a tilting
$\Lambda$-module such that the left multiplication $\Lambda\to\End_{\Lambda} (I_{i})$ is an isomorphism.
Assume that $T = I_{i_{1}}I_{i_{2}} \cdots I_{i_{k-1}}$
is a tilting $\Lambda$-module such that the left multiplication
$\Lambda\to\End_{\Lambda}(T)$ is an isomorphism for $i_{1}, \ldots , i_{k}
 \in \{ 1 , \ldots , n-1 \}$. Then, according to Proposition \ref{3.6}(2),
 we obtain that $TI_{i_{k}}$ is a tilting $\Lambda$-module such that
 the left multiplication $\Lambda\to\End_{\Lambda}(TI_{i_k})$ is an isomorphism.
In particular, $TI_{i_k}$ is basic. Thus we get the assertion inductively.
\end{proof}

By Theorem \ref{3.7}, any element in $\langle I_{1} , \ldots ,
I_{n-1} \rangle$ is a basic tilting $\Lambda$-module. In the
following we show the converse, that is, all basic tilting
$\Lambda$-modules are in $\langle I_{1} , \ldots , I_{n-1}
\rangle$. For this aim, we start with the following.

\begin{proposition}\label{3.8}
Let $T$ be a tilting $\Lambda$-module,
and $ 1 \leq i \leq n-1$. Then we have the following:
\begin{enumerate}[\rm(1)]
\item $\Hom_{\Lambda} (S_{i}, T) = 0$.

\item $\projdim \Hom_{\Lambda} (I_{i}, T) \leq 1$.

\item There exist natural inclusions $T \subseteq \Hom_{\Lambda}
(I_{i}, T) \subseteq T^{\ast \ast} =
\Hom_{\Lambda} (I_{i},T )^{\ast \ast}$.

\item $\Hom_{\Lambda} (I_{i} , T) / T
      \cong \Ext_{\Lambda}^{1} (S_{i} , T)$. If $T
      \subsetneq \Hom_{\Lambda} (I_{i} , T)$, then
      $\Hom_{\Lambda} (I_{i} , T) I_{i} = T$.

\item $\Hom_{\Lambda} (I_{i}, T)$ is a tilting $\Lambda$-module, and
$\End_{\Lambda} (\Hom_{\Lambda} (I_{i},T))\cong
 \End_{\Lambda} (T)$ holds.

\item If $T$ is not a projective $\Lambda$-module, then there
       exists  $1 \leq i \leq n-1$ such that $T
       \subsetneq \Hom_{\Lambda} (I_{i}, T)$.
\end{enumerate}
\end{proposition}

\begin{proof} We firstly note by Lemma \ref{2.6} that $T^{\ast \ast}$ is a projective
$\Lambda$-module. By Lemma \ref{2.4}, we have
 $\Ext_{\Lambda}^{j} (S_{i}, \Lambda) = 0=
 \Ext_{\Lambda}^{j} (S_{i}, T^{\ast\ast})$ for $j\neq2$.
These facts will be used freely in this proof.

(1)  By Lemma \ref{2.7}, we have an exact sequence
\begin{equation}\label{3.8.1}
0 \rightarrow T \xrightarrow{\varphi_{T}} T^{\ast \ast}
 \rightarrow T^{\ast \ast} / T \rightarrow 0.
 \end{equation}
Applying the functor $\Hom_{\Lambda}(S_{i},-)$,
one gets $\Hom_{\Lambda} (S_{i},T) = 0$.

(2) Applying $\Hom_{\Lambda} (-, T^{\ast
 \ast})$ to the short exact sequence
$0 \rightarrow I_{i} \rightarrow \Lambda \rightarrow S_{i}
\rightarrow 0,$
we have an exact sequence
$0 = \Hom_{\Lambda} (S_{i},T^{\ast \ast})
 \rightarrow \Hom_{\Lambda} (\Lambda , T^{\ast \ast}
) \rightarrow \Hom_{\Lambda} (I_{i} , T^{\ast \ast}
) \rightarrow \Ext_{\Lambda}^{1} (S_{i} , T^{\ast
 \ast}) = 0.$
Thus $\Hom_{\Lambda} (I_{i} , T^{\ast \ast}
) \cong T^{\ast \ast}$ is a projective $\Lambda$-module. Then
applying the functor $\Hom_{\Lambda}(I_{i},-)$ to the sequence \eqref{3.8.1}, one gets that
 $\Hom_{\Lambda} (I_{i},T)$ is a submodule of the
 projective $\Lambda$-module $\Hom_{\Lambda} (
 I_{i},T^{\ast \ast})$.
Since $\gldim \Lambda \leq 2$, any submodule of a projective module has projective dimension at most $1$.

(3) Applying $\Hom_{\Lambda}
 (-,T)$ to the exact sequence $0 \rightarrow I_{i} \rightarrow \Lambda \rightarrow S_{i}
\rightarrow 0$ of $(\Lambda,\Lambda)$-bimodules, we obtain an exact
sequence
\begin{equation}\label{3.8.2}
0\rightarrow
 \Hom_{\Lambda} (\Lambda, T) \rightarrow
 \Hom_{\Lambda} (I_{i},T)
\rightarrow\Ext_{\Lambda}^{1} (S_{i},T) \rightarrow 0 \\
 \rightarrow \Ext_{\Lambda}^{1} (I_{i},T)
 \rightarrow \Ext_{\Lambda}^{2} (S_{i},T)
 \rightarrow 0
 \end{equation}
of $\Lambda$-modules by (1). Since the $\Lambda^{\op}$-module $S_i$
is annihilated by $I_i$, the $\Lambda$-module
$\Ext_{\Lambda}^{1}(S_i,T)$ is annihilated by $I_i$ and hence
isomorphic to $S_i^m$ for some $m\ge0$. Hence \eqref{3.8.2} gives an
exact sequence $0 \rightarrow T \rightarrow \Hom_{\Lambda} (I_{i},T)
\rightarrow S_{i}^{m} \rightarrow 0$. Applying $(-)^{\ast} =
\Hom_{\Lambda} (-,\Lambda)$, we obtain an exact sequence $0 =
(S_{i}^{m})^{\ast} \rightarrow \Hom_{\Lambda}
 (I_{i},T)^{\ast} \rightarrow T^{\ast} \rightarrow
 \Ext_{\Lambda}^{1} (S_{i}^{m}, \Lambda) = 0$.
In particular, we have $T^{\ast \ast} \cong \Hom_{\Lambda}
(I_{i},T)^{\ast \ast}$ and the commutative diagram
\[
\xymatrix{
0 \ar[r] & T\ar[r]\ar[d]^{\varphi_{T}} &
\Hom_{\Lambda} (I_{i},T) \ar[r]\ar[d]^{\varphi_{\Hom_{\Lambda}(I_{i},T)}}\ar[r]&S_i^m\ar[r]&0. \\
&T^{\ast \ast} \ar@{=}[r]& \Hom_{\Lambda} (I_{i},T)^{\ast \ast}
}\]
By (2) and Lemma \ref{2.7}, $\varphi_{\Hom_{\Lambda} (I_{i},T)}$ is
 a monomorphism and hence (3) follows.

(4) The former assertion is immediate from the exact sequence \eqref{3.8.2}.
Since $\Ext_{\Lambda}^{1} (S_{i},T)\cong S_i^m$ is annihilated by $I_{i}$, we have $TI_{i} \subseteq
 \Hom_{\Lambda} (I_{i},T) I_{i} \subseteq T$.
For the latter assertion, notice that $\Tor_1^{\Lambda}(T,S_i)\cong\Ext_{\Lambda}^{1} (S_{i},T)\neq0$ by Lemma \ref{3.4}.
Since $T / TI_{i} \cong T \otimes_{\Lambda} S_{i} = 0$ holds by
Lemma \ref{2.9}(3), we obtain $\Hom_{\Lambda} (I_{i} , T) I_{i} =
T$.

(5) If $T = \Hom_{\Lambda} (I_{i} , T)$, then it is obvious. Assume that $T \neq
 \Hom_{\Lambda} (I_{i},T)$.
By (2) and Propositions \ref{3.3}(2) and \ref{2.10}(2), it suffices
to prove that
 $\Ext_{\Lambda}^{j} (I_{i}, T) = 0$ for any $j > 0$.
We only have to consider the case $j=1$ since $\projdim I_{i} \leq 1$.
We have $\Ext_{\Lambda}^{1} (S_{i},T) \neq 0$ by (4), and hence
 $\Ext_{\Lambda}^{1}(I_{i},T) \cong \Ext_{\Lambda}^{2} (S_{i},T) = 0$
 holds by Lemma \ref{3.4}. Thus (5) follows.

(6) By our assumption and Lemma \ref{2.6}, $T \neq T^{\ast \ast}$
 holds. By Lemma \ref{2.7} and Proposition
 \ref{3.1}, we can take a simple submodule $S_{i}$ of
 $T^{\ast \ast} / T$ for some $1\le i\le n-1$.
Applying $\Hom_{\Lambda} (S_{i} , -)$ to the
exact sequence \eqref{3.8.1}, we get an exact sequence
$0 = \Hom_{\Lambda} (S_{i}, T^{\ast \ast})
 \rightarrow \Hom_{\Lambda} (S_{i}, T^{\ast \ast} / T
) \rightarrow \Ext_{\Lambda}^{1} (S_{i}, T)$.
Thus $\Ext_{\Lambda}^{1} (S_{i},T)
\neq 0$ by our choice of $S_{i}$. Thus $\Hom_{\Lambda}
(I_{i}, T ) / T \cong \Ext_{\Lambda}^{1} (S_{i},T) \neq 0$ holds by
(4), and we have $T \subsetneq \Hom_{\Lambda} (I_{i},T)$.
\end{proof}

\begin{lemma}\label{3.9}
Let $T \in \langle I_{1} , \ldots,I_{n-1} \rangle$, and let $f_{T}:T
\rightarrow \Lambda$ be a natural inclusion. Then in the following
commutative diagram, $\varphi_{\Lambda}$
 and $f^{\ast \ast}_{T}$ are isomorphisms.
 \[\xymatrix{
  T  \ar[r]^{\varphi_{T}}\ar[d]^{f_{T}}& T^{\ast \ast} \ar[d]^{f^{\ast \ast}_{T}} \\
  \Lambda  \ar[r]^{\varphi_{\Lambda}}&  \Lambda^{\ast \ast}
}\]
\end{lemma}

\begin{proof} Since $\Lambda$ is projective, it is clear that
$\varphi_{\Lambda}$ is
 an isomorphism.

Any composition factor of the $\Lambda$-module $\Lambda / T$ has a
form $S_{i}$ for some $1 \leq i \leq n-1$. By Lemma \ref{2.4}, we
have $\Ext_{\Lambda}^{j} (\Lambda / T, \Lambda) = 0$ for $j\neq2$.
Applying $(-)^{\ast} = \Hom_{\Lambda} (-,\Lambda)$ to the exact
sequence $0 \rightarrow T \xrightarrow{f_{T}} \Lambda
 \rightarrow \Lambda / T \rightarrow 0,$
we have an exact sequence $ 0 = (\Lambda / T)^{\ast}
\rightarrow \Lambda^{\ast}
 \xrightarrow{f_{T}^{\ast}}T^{\ast} \rightarrow
 \Ext_{\Lambda}^{1} (\Lambda / T , \Lambda) = 0$.
Thus $f_{T}^{\ast}$ is an isomorphism
 and hence $f_{T}^{\ast \ast}$ is an isomorphism.
\end{proof}

Now we are in a position to state our first main result in this
section.

\begin{theorem}\label{3.10}
Let $\Lambda$ be the Auslander algebra of $K[x]/(x^n)$. Then
\begin{enumerate}[\rm(1)]
\item For any tilting $\Lambda$-module $T$, there exists $U \in \langle I_{1} , \ldots ,I_{n-1} \rangle$
such that $\add T = \add U$.
\item If two elements $T$ and $U$ in $\langle I_{1} , \ldots ,I_{n-1} \rangle$ are
isomorphic as $\Lambda$-modules, then $T=U$.
\item The set $\tilt\Lambda$ is given by $\langle I_1,\ldots,I_{n-1}\rangle$.
\item The statements $(1)$, $(2)$ and $(3)$ hold also for $\Lambda^{\op}$-modules.
\end{enumerate}
\end{theorem}

\begin{proof}
(1) By Proposition \ref{3.8}(3)(4)(5)(6), there exists a finite
sequence of tilting
 $\Lambda$-modules
\[T = T_{0} \subsetneq T_{1} \subsetneq \cdots \subsetneq T_{m} = T^{\ast \ast}\]
and $i_1,\ldots,i_m\in\{1,\ldots,n-1\}$ such that $T_{k+1} = \Hom_{\Lambda} (I_{i_{k+1}} , T_{k})$
 and $T_{k} = T_{k+1} I_{i_{k+1}}$ for any $0\le k\le m-1$.
In particular, we have $T = T_{1}I_{i_{1}} = T_{2}I_{i_{2}}I_{i_{1}}
= \cdots = T_{m}I_{i_{m}} \cdots I_{i_{1}}$.
Because $T^{\ast \ast}$ is a projective tilting $\Lambda$-module by
 Lemma \ref{2.6}, we have $\add T_{m} = \add
 \Lambda$. Thus $\add T = \add  U$ holds for $U:= I_{i_{m}} \cdots I_{i_{1}}\in
 \langle I_{1} , \ldots ,I_{n-1} \rangle$.

(2) For $T,U \in \langle I_{1} , \ldots ,I_{n-1}
 \rangle$, assume that there exists a $\Lambda$-module isomorphism $g:T \cong U$.

By Lemma \ref{3.9}, there exists a commutative diagram
\[\xymatrix{
&T\ar[r]^g_\sim\ar[d]^{\varphi_T}\ar[dl]_{f_T}&U\ar[d]^{\varphi_T}\ar[dr]^{f_U}\\
\Lambda&T^{\ast\ast}\ar[r]_{g^{\ast\ast}}^\sim\ar[l]^{e_T}&U^{\ast\ast}\ar[r]_{e_U}&\Lambda
}\]
where $e_{T} := \varphi_{\Lambda}^{-1}f_{T}^{\ast \ast}$ and $e_{U} :=
 \varphi_{\Lambda}^{-1}f_{U}^{\ast \ast}$ are isomorphisms.
Putting $h = e_{U}g^{\ast \ast}e_{T}^{-1}:\Lambda \rightarrow
 \Lambda$, we have a commutative diagram
\[\xymatrix{T\ar[r]^g\ar[d]^{f_T}&U\ar[d]^{f_U}\\
\Lambda\ar[r]^\sim_h&\Lambda. }\] Since $h$ is given by the left
multiplication of an invertible element $x\in\Lambda$, so is $g$.
Since $T$ is an ideal of $\Lambda$, we have $U=xT=T$.

(3) This is a consequence of (1), (2) and Theorem \ref{3.7}.

(4) One can prove it similarly to (1), (2) and (3).
\end{proof}

The mutations of tilting $\Lambda$-modules are described by the
following result. Notice that we use the structure of
$\Lambda^{\op}$-modules when we consider mutations of
$\Lambda$-modules.

\begin{proposition}\label{3.11}
Let $T \in \langle I_{1} , \ldots ,I_{n-1}
\rangle$.
 \begin{enumerate}[\rm(1)]

\item For each $1 \leq i \leq n-1$, precisely one of the following
statements {\rm(a)} and {\rm(b)} holds.
\begin{itemize}

  \item[\rm(a)] $I_iT\neq T$ and $\Hom_{\Lambda^{\op}}(I_{i},T) = T$ hold,
  and $I_{i}T = I_{i} \otimes_{\Lambda} T$ is a left mutation of $T$ at $e_iT$.

  \item[\rm(b)] $I_{i}T = T$ and $\Hom_{\Lambda^{\op}}(I_{i}, T)\neq T$ hold,
  and $\Hom_{\Lambda^{\op}} (I_{i}, T)$ is a right mutation of $T$ at $e_iT$.
\end{itemize}

\item All mutations of $T$ in $\tilt\Lambda$ are of the form
       $(1)$. In particular, $T$ has precisely $n-1$
      mutations in $\tilt\Lambda$.

\item The corresponding statements to $(1)$ and $(2
     )$ hold for $\Lambda^{\op}$-modules.
\end{enumerate}
\end{proposition}

\begin{proof}
(1) Applying Proposition \ref{3.6}(2) and
Proposition \ref{3.8}(5) to the tilting $\Lambda^{\op}$-module $T$, we have
 that $I_{i}T$ and $\Hom_{\Lambda^{\op}} (I_{i} , T)$
 are tilting $\Lambda^{\op}$-modules with $\End_{\Lambda^{\op}} (
 I_{i}T) \cong \End_{\Lambda^{\op}} (T)\cong\End_{\Lambda^{\op}} (
 \Hom_{\Lambda^{\op}} (I_{i},T ))$.
Since $\End_{\Lambda^{\op}} (T)\cong
 \Lambda^{\op}$ holds by Theorem \ref{3.7}, we have that $I_{i}T$ and
 $\Hom_{\Lambda^{\op}} (I_{i},T)$ are tilting
 $\Lambda$-modules. Further we know that
\begin{equation*}
 I_{i}T  = \bigoplus_{j=1}^{n} e_{j} I_{i}T\ \mbox{ and }\
 \Hom_{\Lambda^{\op}} (I_{i}, T)  =
 \bigoplus_{j=1}^{n} \Hom_{\Lambda^{\op}} (I_{i}e_{j} , T).
\end{equation*}
Since $e_{j} I_{i} = e_{j} \Lambda$ and $I_{i}e_{j} = \Lambda e_{j}$
 hold for any $j \neq i$, the indecomposable direct summands of
 $I_{i}T$ (resp. $\Hom_{\Lambda^{\op}} (I_{i} , T)$)
 coincide with those of $T$ except one. By Theorem
 \ref{2.m}, $I_{i}T$ (resp. $\Hom_{\Lambda^{\op}}(I_{i} , T)$)
 is either isomorphic to $T$ or a mutation of $T$. We have
\begin{equation*}
\begin{aligned}
 I_{i}T \cong T  \quad \Longleftrightarrow \quad S_{i} \otimes_{\Lambda} T
 = 0 \quad & \Longleftrightarrow \quad \Hom_{\Lambda^{\op}} (
 S_{i} , T) = 0, \\
 \Hom_{\Lambda^{\op}} (I_{i} , T) \cong T \quad
 & \Longleftrightarrow \quad
 \Ext_{\Lambda^{\op}}^{1} (S_{i} , T) = 0
\end{aligned}
\end{equation*}
by Proposition \ref{3.8}. Thus precisely one of $S_{i} \otimes_{\Lambda}
T = 0$ and
 $\Tor_{1}^{\Lambda} (S_{i} , T) = 0$ holds by
 Proposition \ref{2.10}.

It remains to decide whether the mutation is left or right. We only
have
 to show $\Hom_{\Lambda^{\op}} (I_{i} , T) \geq T
 \geq I_{i}T$. Taking an epimorphism $\Lambda^{m} \rightarrow I_{i}$ of
 $\Lambda$-modules, we have an epimorphism $T^{m} \rightarrow I_{i}T$.
Thus, we have $T^{\bot} \supseteq (
 I_{i}T)^{\bot}$ and $T \geq I_{i}T$.
If $U := \Hom_{\Lambda^{\op}} (I_{i}, T)
 \supsetneq T$, then we have $I_{i}U = T$ by Proposition \ref{3.8}.
 Thus we have $\Hom_{\Lambda^{\op}} (I_{i}, T ) = U \geq T$.

(2) Any basic tilting $\Lambda$-module has precisely $n$
indecomposable direct summands. Since $P_{n}$ is injective by
Proposition \ref{3.1}, it is a direct summand of any tilting
$\Lambda$-module. Therefore the number of mutations of $T$ in
$\tilt\Lambda$ is at most $n-1$, while we have at least $n-1$
mutations in $\tilt\Lambda$ by (1).

(3) One can prove it similarly to (1) and (2).
\end{proof}

Immediately we have the following description of the Hasse quiver of
tilting $\Lambda$-modules.

\begin{corollary}\label{3.12}
The Hasse quiver of $\tilt\Lambda$ has the set
$\langle I_{1},\ldots ,I_{n-1}\rangle$ of vertices.
All arrows starting or ending at
$T\in\langle I_{1} , \ldots ,I_{n-1}\rangle$ are given by
\begin{eqnarray*}
\mu_i(T):=\Hom_{\Lambda^{\op}} (I_{i}  ,T)  \longrightarrow T
 &  \mbox{ if } & T = I_{i}T, \\
  T  \longrightarrow \mu_i(T):=I_{i}T & \mbox{ if } &T \neq I_{i}T
\end{eqnarray*}
for each $1\le i\le n-1$, where $\mu_i(T)$ is the mutation of $T$
at the direct summand $e_iT$ (Definition \ref{mutation}).
Thus the number of arrows starting or ending at $T$ is precisely $n-1$.
\end{corollary}

We have shown that the set $\tilt\Lambda$ is
given by $\langle I_{1} , \ldots , I_{n-1} \rangle$.
In the following we give an explicit description of this set.
Let us start with the following elementary observation.

\begin{proposition}\label{3.14}
Let $A$ be a basic finite dimensional algebra, $\{e_{1} ,
\ldots , e_{n} \}$ a complete set of orthogonal primitive idempotents of
 $A$, and $S_1,\ldots,S_n$ the corresponding simple $A$-modules.
For a subset $J$ of $\{ 1, \ldots , n \}$, we put
\[
  e_{J} = \sum_{i \in J} e_{i} \quad \text{and}
 \quad I_{J} = A(1- e_{J})A.
\]
Then for any $X \in \mod A$, we have that $XI_{J}$ is the
minimum amongst submodules $Y$ of $X$ satisfying the following
condition:
\begin{itemize}
\item[$(\sharp)$] Any composition factor of $ X/Y$ has the form
$S_{i}$ for some $i \in J$.
\end{itemize}
\end{proposition}

\begin{proof}
Since $\Hom_{A}((1 - e_{J})A , X) \cong X
 (1 - e_{J})$, we have
\[
XI_{J}=X(1-e_{J})A= \sum_{ f \in \Hom_{A} ((1 -
 e_{J})A , X ) } \mathrm{Im}f.
\]
The condition $(\sharp)$ holds if and only
if $\Hom_{A} ((1 - e_{J})A, X/Y ) = 0$ holds if and only if
$\mathrm{Im}f \subseteq Y$ holds for any $f \in \Hom_{A} ((1-e_{J})A ,X)$
if and only if $ XI_{J}\subseteq Y$.
\end{proof}

We have the following relations for the multiplication of ideals $I_{1} , \ldots , I_{n-1}$.

\begin{proposition}\label{3.15} Let $I_i$ be the maximal ideal of $\Lambda$ as above. Then the following relations hold for any $1 \leq i,j \leq n-1$.
\begin{enumerate}[\rm(1)]

\item $I_{i}^{2} = I_{i}$.

\item If $ \left|  i-j \right| \geq 2$, then $I_{i} I_{j} = I_{j}
I_{i}$.

\item If $|i-j|=1$, then $I_{i} I_{j} I_{i} = I_{j} I_{i} I_{j}$.
 \end{enumerate}
\end{proposition}

\begin{proof}
(1) By Propositions \ref{3.14} and \ref{3.1}, $I_{i} = \Lambda (1-e_{i}) \Lambda$ holds.
Hence $I_{i}^{2} = \Lambda (1 - e_{i}) \Lambda(1 - e_{i}) \Lambda =
\Lambda (1-e_{i}) \Lambda = I_{i}$.

(2)(3) For $1 \leq i \neq j \leq n-1$, put $I_{i,j} = \Lambda ( 1 - e_{i} - e_{j})\Lambda$.
Removing all vertices except $i$ and $j$ from the quiver with relations
 of $\Lambda$, we have the quiver with relations of $\Lambda / I_{i,j}$.
In particular, if $\left| i - j \right| \geq 2$, then
$\Lambda/I_{i,j} \cong K \times K$. If $\left| i - j \right| = 1$,
then $\Lambda/I_{i,j}$ is given by the quiver
 $\xymatrix{i\ar@<.2em>[r]^a&j\ar@<.2em>[l]^b}$ with relations $ab = 0 = ba$ and hence
$\Lambda/I_{i,j}=\left[\begin{smallmatrix} i\\ &j\end{smallmatrix}\middle|
\begin{smallmatrix} &j\\ i\end{smallmatrix}\right]$.

We prove (2). By Proposition \ref{3.14}, $I_{i}I_{j} \supseteq I_{i,j}$.
Since $\Lambda / I_{i,j} \cong K \times K$, we have $I_{i}I_{j}/I_{i,j}=0$.
Hence $I_{i}I_{j} = I_{i,j}$ holds, and similarly we have $I_{j}I_{i} = I_{i,j}$.
Thus $I_{i}I_{j} = I_{i,j} = I_{j}I_{i}$.

We prove (3). By Proposition \ref{3.14}, $I_{i}I_{j}I_{i} \supseteq I_{i,j}$.
Since $\Lambda / I_{i,j}=\left[\begin{smallmatrix} i\\ &j\end{smallmatrix}\middle|
\begin{smallmatrix} &j\\ i\end{smallmatrix}\right]$, we have $I_{i}I_{j}I_{i}/I_{i,j}=0$.
Hence $I_{i}I_{j}I_{i} = I_{i,j}$ holds, and similarly we have
 $I_{j}I_{i}I_{j} = I_{i,j}$. Thus $I_{i}I_{j}I_{i} = I_{i,j} = I_{j}I_{i}I_{j}$.
\end{proof}

Now we recall some well-known properties of the symmetric groups. We
consider the action of $\mathfrak{S}_n$ on $\mathbb{R}^n$ given by
permuting the standard basis $e_1,\ldots,e_n$. Then $\mathfrak{S}_n$
acts on the subspace
\[V:=\{x_1e_1+\cdots+x_ne_n\in\mathbb{R}^n\mid\sum_{i=1}^nx_i=0\},\]
which has a basis $\alpha_i:=e_i-e_{i+1}$ with $1\le i\le n-1$.
Clearly the action of $\mathfrak{S}_n$ on $V$ is faithful, and we
have an injective homomorphism $\mathfrak{S}_n\to{\rm GL}(V)$ called the
\emph{geometric representation}.

Let $s_i$ be the transposition $(i,i+1)\in \mathfrak{S}_{n}$.
The following elementary fact  plays an important role in the proof of our main theorem.

\begin{proposition}\label{3.16}
Let $\mathfrak{S}_{n}$ be the symmetric
group of degree $n$ and $\mathfrak{S}_{n} \ni w$. Then we have the
following:
\begin{enumerate}[\rm(1)]
\item \cite[Theorem 3.3.1]{BjB} Any expression $s_{i_{1}}s_{i_{2}} \cdots
 s_{i_{l}}$ of $w$ can be transformed into a reduced expression of $w$
       by applying the following operations {\rm(a)}, {\rm(b)}, {\rm(c)} repeatedly.
\begin{itemize}
\item[\rm(a)] Remove $s_{i}s_{i}$ in the expression.

\item[\rm(b)] Replace $s_{i} s_{j}$ with $\left| i
- j\right| \geq 2$ by $s_{j} s_{i}$ in the expression.
\item[\rm(c)] Replace $s_{i}s_{j}s_{i}$ with $|i-j|=1$ by
       $s_{j}s_{i}s_{j}$ in the expression.
\end{itemize}
\item \cite[Theorem 3.3.1]{BjB} Every two reduced expressions of $w$ can be transformed into each
other by applying the operations {\rm(b)} and {\rm(c)} repeatedly.
%\item \new{\cite[4.2.7]{BjB} The geometric representation $\mathfrak{S}_n\to{\rm GL}(V)$ is faithful.}
\item If $w=s_{i_1}s_{i_2}\cdots s_{i_l}$ is a reduced expression, then $s_{i_1}\cdots s_{i_k}(\alpha_{i_{k+1}})$ is a positive root for any $1\le k\le l-1$.
\end{enumerate}
\end{proposition}

%We also need the following well-known result [BB, Theorem $4.2.7$].

%\vspace{0.2cm}

%\noindent{\bf Proposition 3.17} {\it A group homomorphism $f$ $\mathrm{:}$
%$\mathfrak{S}_{n} \rightarrow
 %\mathrm{GL}_{n-1} (\mathbb{R})$ is given by
%\begin{equation*}
%   (e_{j})^{s_{i}} =
 %   \left\{
  %    \begin{matrix}
 %  -e_{i} + e_{i-1} + e_{i+1}  &  (i=j) \\
 %    e_{j} & (i \neq j )
  %    \end{matrix}
  %  \right.
%\end{equation*}
%for $1 \leq i,j \leq n-1$. Moreover, this is injective.}

We also need the following proposition.

\begin{proposition}\label{3.m}

There exists a well-defined surjective map $\mathfrak{S}_{n}
\rightarrow \langle I_{1} , \ldots, I_{n-1} \rangle$ which maps $w$
to $I(w) = I_{i_{1}}\cdots I_{i_{l}}$, where $w = s_{i_{1}} \cdots
s_{i_{l}}$ is an arbitrary reduced expression.
\end{proposition}

\begin{proof}
First, we show that the map is well-defined. Take two reduced
expressions $w = s_{i_{1}} \cdots s_{i_{l}} = s_{j_{1}} \cdots
s_{j_{l}}$
 of $w$. These two expressions are transformed into each other by the operation
 (b) and (c) in Proposition \ref{3.16}. Then by Proposition \ref{3.15},
 we obtain $I_{i_{1}} \cdots I_{i_{l}} = I_{j_{1}} \cdots I_{j_{l}}$.

Next we show that the map is surjective. For any  $I \in \langle
I_{1}, \ldots , I_{n-1} \rangle$, we take a minimal number $l$ such
that $I = I_{i_{1}} \cdots I_{i_{l}}$ holds for some $i_{1}, \cdots
, i_{l}\in\{1,\ldots,n-1\}$. Now we put $w := s_{i_{1}} \cdots
s_{i_{l}}$. This expression is
 transformed into a reduced expression of $w$ by applying (a), (b) and  (c) in Proposition \ref{3.16}.
Since $l$ is minimal, then (a) would not happen. Therefore $w =
s_{i_{1}} \ldots s_{i_{l}}$ is a reduced expression and we have $I =
I(w)$.
\end{proof}

Since $I(w)$ is a tilting $\Lambda$-module with
$\End_\Lambda(I(w))\cong\Lambda$ for any $w\in\mathfrak{S}_n$ by
Proposition \ref{3.m}, we have an autoequivalence
\[-\LotimesL I(w):\DD^{\bo}(\mod\Lambda)\to \DD^{\bo}(\mod\Lambda)\]
whose quasi-inverse is given by $\RHom_\Lambda(I(w),-)$.
We define a full subcategory $\TT$ of $\DD^{\bo}(\mod\Lambda)$ by
\[\TT:=\{X\in\DD^{\bo}(\mod\Lambda)\mid\forall i\in\mathbb{Z}\ H^i(X)e_n=0\}.\]
The Grothendieck group $K_0(\TT)$ is a free abelian group with basis $[S_1],\ldots,[S_{n-1}]$.
We identify $V$ with $\mathbb{R}\otimes_{\mathbb{Z}}K_0(\TT)$ by $\alpha_i=[S_i]$ for any $1\le i\le n-1$.
%we have an action of $\mathfrak{S}_n$ on $\mathbb{R}\otimes_{\mathbb{Z}}K_0(\TT)$.

\begin{lemma}\label{restriction to T}
\begin{enumerate}[\rm(1)]
\item We have an induced autoequivalence $-\LotimesL I(w):\TT\to\TT$.
\item We have $[-\LotimesL I_i]=s_i$ in ${\rm GL}(V)$ for any $1\le i\le n-1$.
%\left\{\begin{array}{cc}
%-[S_i]+[S_{i-1}]+[S_{i+1}] & \mbox{if }\ i=j\\
%{[S_j]} & \mbox{if }\ i\neq j
%\end{array}\right.$ where $[S_0]:=0$ by convention.
\end{enumerate}
\end{lemma}

\begin{proof}
(1) We have a triangle $I(w)\to\Lambda\to\Lambda/I(w)\to I(w)[1]$ in $\DD(\Mod\Lambda^{\op}\otimes_K\Lambda)$.
Applying $X\LotimesL-$ for $X\in\TT$, we have a triangle
\begin{equation}\label{3.17.1}
X\LotimesL I(w)\to X\to X\LotimesL(\Lambda/I(w))\to X\LotimesL I(w)[1]
\end{equation}
in $\DD^{\bo}(\mod\Lambda)$.
Since both $X$ and $X\LotimesL(\Lambda/I(w))$ belong to $\TT$, so is $X\LotimesL I(w)$.
Thus $\TT\LotimesL I(w)\subseteq\TT$ holds.
Similarly one can show $\RHom_\Lambda(I(w),\TT)\subseteq\TT$. Therefore the assertion follows.

(2) For $X\in\DD^{\bo}(\mod\Lambda)$ and $Y\in\DD^{\bo}(\mod\Lambda^{\op})$, let
$\chi(X,Y):=\sum_{k\in\mathbb{Z}}(-1)^k\dim_KH^k(X\LotimesL Y)$. Then
\[\chi(S_j,S_i)=\left\{\begin{array}{cc}
2&i=j\\
-1&|i-j|=1\\
0&|i-j|\ge2
\end{array}\right.\]
holds for any $1\le j\le n-1$. We have $[S_j\LotimesL
I_i]=[S_j]-[S_j\LotimesL S_i]=[S_j]-\chi(S_j,S_i)[S_i]$ by applying
\eqref{3.17.1} to $X=S_j$ and $w=s_i$. Thus the assertion follows
easily.
\end{proof}

We have the following key observations.

\begin{proposition}\label{two actions same}
Let $w\in\mathfrak{S}_n$ and $w=s_{i_1}s_{i_2}\cdots s_{i_l}$ a reduced expression.
\begin{enumerate}[\rm(1)]
\item We have $[-\LotimesL I(w)]=w^{-1}$ in ${\rm GL}(V)$.
\item We have $I_{i_l}\supsetneq I_{i_{l-1}}I_{i_l}\supsetneq\cdots\supsetneq I_{i_1}\cdots I_{i_l}$ and
$I(w)=I_{i_1}\LotimesL\cdots\LotimesL I_{i_l}$.
\item Let $1\leq j\leq n-1$. Then $l(s_jw)>l(w)$ if and only if
$I(s_jw)<I(w)$.
\end{enumerate}
\end{proposition}

\begin{proof}
The assertion (2) implies (1) since Lemma \ref{restriction to T}(2) implies
$[-\LotimesL I(w)]=[-\LotimesL I_{i_l}]\circ\cdots\circ[-\LotimesL I_{i_2}]\circ[-\LotimesL I_{i_1}]
=s_{i_l}\cdots s_{i_2}s_{i_1}=w^{-1}$.

We prove (2) inductively. This is clear for $l=1$. For
$u:=s_{i_2}\cdots s_{i_{l}}$, we assume $I_{i_l}\supsetneq
I_{i_{l-1}}I_{i_l}\supsetneq\cdots\supsetneq I_{i_2}\cdots
I_{i_{l}}$ and $I(u)=I_{i_2}\LotimesL\cdots\LotimesL I_{i_{l}}$.
Then $[S_{i_1}\LotimesL I(u)]=u^{-1}(\alpha_{i_1})=s_{i_l}\cdots
s_{i_{2}}(\alpha_{i_1})$ is a positive root by Proposition
\ref{3.16}(3). Hence $S_{i_1}\otimes_\Lambda I(u)\neq0$ holds, and
we have $I(u)\supsetneq I_{i_1}I(u)=I(w)$. Thus $I_{i_1}\LotimesL
I(u)=I(w)$ holds by Proposition \ref{3.6}(1), and the assertion
follows.

(3) It suffices to show that $l(s_{j}w)>l(w)$ implies that
$I(s_{j}w)<I(w)$ by replacing $s_{j}w$ with $w$ if necessary. By (2)
we have $I(w)\supsetneq I(s_{j}w)=I_{j}I(w)$. Then by Proposition
\ref{3.11}(1)(a), we have $I(s_{i}w)<I(w)$.
\end{proof}

Now we have the following main result in this section.

\begin{theorem}\label{3.17}
\begin{enumerate}[\rm(1)]
\item There exists a well-defined bijection $\mathfrak{S}_{n} \cong
\langle I_{1} , \ldots, I_{n-1} \rangle$ which maps $w$ to $I(w) =
I_{i_{1}}\cdots I_{i_{l}}$, where $w = s_{i_{1}} \cdots s_{i_{l}}$ is
an arbitrary reduced expression.

\item Consequently, there exists a bijection  $I:\mathfrak{S}_{n} \cong\tilt\Lambda$. In particular $\#\tilt\Lambda=n!$.
\item The bijection $I$ in (2) is an anti-isomorphism of posets with respect to the left order on $\mathfrak{S}_{n}$ and the generation order on $\tilt\Lambda$.
\end{enumerate}
\end{theorem}

\begin{proof}
(1) By Proposition \ref{3.m}, $I$ is a well-defined surjective map.
Now we show that the map is injective. If $I(w)=I(w^\prime)$, then
$[-\LotimesL I(w)]=[-\LotimesL I(w^\prime)]$ in ${\rm GL}(V)$. By
Proposition \ref{two actions same}(1), the images of $w$ and
$w^\prime$ in ${\rm GL}(V)$ are the same. Since
$\mathfrak{S}_n\to{\rm GL}(V)$ is injective, we have $w=w'$.

(2) This is immediate from (1) and Theorem \ref{3.10}(3).

(3) In the Hasse quiver of the left order on $\mathfrak{S}_{n}$,
arrows ending at $w \in \mathfrak{S}_{n}$ are given by $w\to s_{i}w$
with $1 \leq i \leq n-1$ satisfying $l(s_{i}w) > l(w)$. By Proposition \ref{two
actions same}(3), the Hasse quiver of $\tilt\Lambda$ coincides with
the opposite of the Hasse quiver of $\mathfrak{S}_{n}$. Thus $I$ is
an anti-isomorphism by Lemma \ref{2.15}.
\end{proof}

Immediately we have the following corollary.

\begin{corollary}\label{3.a} For any expression $w=s_{i_1}s_{i_2}\cdots
s_{i_l}\in\mathfrak{S}_n$,
 $I(w)=\mu_{i_1}\mu_{i_2}\cdots\mu_{i_l}(\Lambda)$ holds,
 where $\mu_i$ is defined in Corollary \ref{3.12}.
\end{corollary}

\begin{proof}
It suffices to show that, if $l(s_iw)=l(w)+1$, then
$I(s_iw)=\mu_i(I(w))$ holds. Since $I(s_iw)\not\cong I(w)$ holds by
Proposition 3.16(2), the assertion follows from Theorem 3.10(1)(a).
\end{proof}

To compare with the Hasse quiver of tilting $\Lambda$-modules, we
  give the Hasse quiver of the left order on the
symmetric group $\mathfrak{S}_{n}$ for $n=2,3$.

\begin{example}\label{3.18}
We describe the Hasse quiver of the left order on $\mathfrak{S}_{2}$
and $\mathfrak{S}_{3}$.
\begin{enumerate}[\rm(1)]
\item The Hasse quiver of the left order on
 $\mathfrak{S}_{2}$ is the opposite of the following quiver:
$$\begin{xy}
 (0,0)*+{{\rm id}=\begin{smallmatrix} [12]\\
\end{smallmatrix}}="1",
(25,0)*+{\begin{smallmatrix} [21]\\
\end{smallmatrix}=s_1}="2",
\ar"1";"2"
\end{xy}$$

\item The Hasse quiver of the left order on
 $\mathfrak{S}_{3}$ is the opposite of the following quiver:
$$\begin{xy}
0;<4pt,0pt>:<0pt,2.2pt>::
 (0,0)*+{{\rm id}=\begin{smallmatrix} [123]\\
\end{smallmatrix}}="1",
(-25,-6)*+{s_1=\begin{smallmatrix} [213]\\
\end{smallmatrix}}="2",
(25,-6)*+{\begin{smallmatrix} [132]\\
\end{smallmatrix}=s_2}="3",
(-25,-18)*+{s_2s_1=\begin{smallmatrix} [312]\\
\end{smallmatrix}}="4",
(25,-18)*+{\begin{smallmatrix} [231]\\
\end{smallmatrix}=s_1s_2}="5",
(0,-24)*+{s_1s_2s_1=\begin{smallmatrix} [321]\\
\end{smallmatrix}=s_2s_1s_2}="6",
\ar"1";"2", \ar"1";"3", \ar"2";"4", \ar"3";"5", \ar"4";"6",
\ar"5";"6",
\end{xy}$$
\end{enumerate}
\end{example}

By Corollary \ref{3.12}, we can describe the Hasse quiver of
 tilting modules over the Auslander algebra $\Lambda$ of $K[x]/(x^{n})$ for $n=2,3$.

\begin{example}\label{3.13}
Denote by $\Lambda_i$ the Auslander algebra of $K[x]/(x^i)$ for
$i=2,3$. Then we have
\begin{enumerate}[\rm(1)]
\item The Hasse quiver ${\rm H}(\tilt\Lambda_{2})$ is the following:
$$\begin{xy}
(0,0)*+{\Lambda_2= \left[\begin{smallmatrix} 1\\ &2\end{smallmatrix}\middle| \begin{smallmatrix} &2\\ 1\\
&2\end{smallmatrix}\right]}="1",
(30,0)*+{I_1=\left[\begin{smallmatrix} &2\end{smallmatrix}\middle| \begin{smallmatrix} &2\\ 1\\
&2\end{smallmatrix}\right]}="2", \ar"1";"2",
\end{xy}$$

\item The Hasse quiver ${\rm H}(\tilt\Lambda_{3})$ is the following:
$$\begin{xy}
0;<5pt,0pt>:<0pt,3pt>::
(0,0)*+{\Lambda_3= \left[\begin{smallmatrix} 1\\ &2\\&&3\end{smallmatrix}\middle| \begin{smallmatrix} &2\\ 1&&3\\
&2\\&&3\end{smallmatrix}\middle| \begin{smallmatrix} &&3\\&2\\ 1&&3\\
&2\\&&3\end{smallmatrix}\right]}="1",
(-30,-10)*+{I_1= \left[\begin{smallmatrix} &2\\&&3\end{smallmatrix}\middle| \begin{smallmatrix} &2\\ 1&&3\\
&2\\&&3\end{smallmatrix}\middle| \begin{smallmatrix} &&3\\&2\\ 1&&3\\
&2\\&&3\end{smallmatrix}\right]}="2",
(30,-10)*+{\left[\begin{smallmatrix} 1\\ &2\\&&3\end{smallmatrix}\middle| \begin{smallmatrix} 1&&3\\
&2\\&&3\end{smallmatrix}\middle| \begin{smallmatrix} &&3\\&2\\ 1&&3\\
&2\\&&3\end{smallmatrix}\right]=I_2}="3",
(-30,-25)*+{I_2I_1= \left[\begin{smallmatrix}  &2\\&&3\end{smallmatrix}\middle| \begin{smallmatrix} &&3\\
&2\\&&3\end{smallmatrix}\middle| \begin{smallmatrix} &&3\\&2\\ 1&&3\\
&2\\&&3\end{smallmatrix}\right]}="4",
(30,-25)*+{\left[\begin{smallmatrix} &&3\end{smallmatrix}\middle| \begin{smallmatrix}  1&&3\\
&2\\&&3\end{smallmatrix}\middle| \begin{smallmatrix} &&3\\&2\\ 1&&3\\
&2\\&&3\end{smallmatrix}\right]=I_1I_2}="5",
(0,-35)*+{I_1I_2I_1=\left[\begin{smallmatrix} &&3\end{smallmatrix}\middle| \begin{smallmatrix}  &&3\\
&2\\&&3\end{smallmatrix}\middle| \begin{smallmatrix} &&3\\&2\\ 1&&3\\
&2\\&&3\end{smallmatrix}\right]=I_2I_1I_2}="6",
\ar"1";"2",\ar"1";"3",\ar"2";"4",\ar"3";"5", \ar"4";"6",\ar"5";"6",
\end{xy}$$
\end{enumerate}
\end{example}

\section{Support $\tau$-tilting modules over the Auslander algebra of $K[x]/(x^n)$}

Throughout this section, $\Lambda$ is the Auslander algebra of
$K[x]/(x^n)$. In this section, we firstly construct a bijection from
the symmetric group $\mathfrak{S}_{n+1}$ to the set $\sttilt\Lambda$
of isomorphism classes of basic support $\tau$-tilting
$\Lambda$-modules, and then we show that this is an anti-isomorphism
of posets. Recall that $\Lambda$ is presented by the quiver
\[\xymatrix{
1\ar@<2pt>[r]^{a_1}&2\ar@<2pt>[r]^{a_2}\ar@<2pt>[l]^{b_2}&3\ar@<2pt>[r]^{a_3}\ar@<2pt>[l]^{b_3}&\cdots\ar@<2pt>[r]^{a_{n-2}}\ar@<2pt>[l]^{b_4}&n-1\ar@<2pt>[r]^{a_{n-1}}\ar@<2pt>[l]^{b_{n-1}}&n\ar@<2pt>[l]^{b_n}
}\]
with relations $a_{1} b_{2}= 0$ and $a_{i} b_{i+1} =b_{i}a_{i-1}$ for any $2 \leq i \leq n-1$.
Let $M$ be the ideal of $\Lambda$ generated by $e_n$, and $\overline{\Lambda}:=\Lambda/M$. Then we
have $M=\bigoplus_{i=1}^nM_i$, where $M_i=e_i M$. We often use the functor
\[\overline{(\ )}:=-\otimes_\Lambda\overline{\Lambda}:\mod\Lambda\to\mod\overline{\Lambda}.\]
For example, $\Lambda$ and $M$ in the case $n=4$ are the following.
\begin{eqnarray*}
M=\left[\begin{smallmatrix} \ \\ &\ \\
&&\ \\&&&4 \\ \end{smallmatrix}\middle| \begin{smallmatrix} &\ \\ \ &&\ \\
&\ &&4\\&&3 \\&&&4\end{smallmatrix}\middle| \begin{smallmatrix} &&\\ &\ &&4\\
\ &&3\\&2&&4\\&&3\\&&&4\\ \end{smallmatrix}\middle| \begin{smallmatrix}&&&4 \\&&3\\ &2&&4\\
1&&3\\&2&&4\\&&3\\&&&4\\ \end{smallmatrix}\right]
&\subseteq&\Lambda=\left[\begin{smallmatrix} 1\\ &2\\
&&3\\&&&4\\ \end{smallmatrix}\middle| \begin{smallmatrix} &2\\ 1&&3\\
&2&&4\\&&3\\&&&4\end{smallmatrix}\middle| \begin{smallmatrix} &&3\\ &2&&4\\
1&&3\\&2&&4\\&&3\\&&&4\\ \end{smallmatrix}\middle| \begin{smallmatrix}&&&4 \\&&3\\ &2&&4\\
1&&3\\&2&&4\\&&3\\&&&4\\ \end{smallmatrix}\right]
\end{eqnarray*}

We start with some facts on $\mathfrak{S}_{n+1}$. We
denote by $s_{i}$ the transposition $(i,i+1)$ in
$\mathfrak{S}_{n+1}$ for $1\le i\le n$. Now we prepare the
following, which will be used later.

\begin{lemma}\label{4.1}
\begin{enumerate}[\rm(1)]
\item $\mathfrak{S}_{n+1}=\bigsqcup_{i=0}^n s_{i+1}\cdots s_n
\mathfrak{S}_{n}$, where
$s_{i+1}\cdots s_n\mathfrak{S}_{n}=\mathfrak{S}_n$ for $i=n$.
\item Let $v\in\mathfrak{S}_n$, $1\le i\le n$ and $w=s_{i+1}\cdots s_nv\in\mathfrak{S}_{n+1}$.
\begin{enumerate}[\rm(a)]
%\item $l(w)=n-i+l(v)$.
\item If $j\le i-1$, then $s_jw=s_{i+1}\cdots s_ns_jv$.
\item If $j\ge i+2$, then $s_jw=s_{i+1}\cdots s_ns_{j-1}v$.
\end{enumerate}
\end{enumerate}
\end{lemma}

\begin{proof}
(1) An element $w\in\mathfrak{S}_{n+1}$ belongs to $s_{i+1}\cdots
s_n\mathfrak{S}_{n}$ if and only if $w(n+1)=i+1$ holds. Thus the
assertion follows.

(2) (a) is clear. (b) follows from $s_jw=s_{i+1}\cdots
s_{j-2}s_js_{j-1}s_j\cdots s_nv= s_{i+1}\cdots
s_{j-1}s_{j}s_{j-1}s_{j+1}\cdots\\ s_nv=s_{i+1}\cdots s_ns_{j-1}v$.
\end{proof}

By Lemma \ref{4.1}, elements in $\mathfrak{S}_{n+1}$ are obtained from elements in $\mathfrak{S}_n$ by multiplying $s_{i+1}\cdots s_n$.
Similarly, we will construct support $\tau$-tilting $\Lambda$-modules from tilting $\Lambda$-modules by applying successive mutations.

In the rest, for $T\in\langle I_1,\ldots,I_{n-1}\rangle$, we consider a direct sum decomposition
\[T=\bigoplus_{i=1}^nT_i\ \mbox{ for }\ T_i:=e_iT.\]
We need the following observations on these direct summands.

\begin{lemma}\label{4.3}
Let $T\in\langle I_1,\ldots,I_{n-1}\rangle$. For any $1\leq i\leq n$, we have
\begin{enumerate}[\rm(1)]
\item $\soc T_i\cong S_n$.
\item $\overline{T_i}$ is either zero or indecomposable with a simple socle $S_{n-i}$.
\item $\overline{T_i}$ has no composition factors isomorphic to $S_{n}$. In particular $\Hom_{\Lambda}(\overline{T_i}, T)=0$.
\item Let $V\in\langle I_1,\ldots,I_{n-1}\rangle$. If $\overline{T_i}\cong\overline{V_i}$, then $T_i\cong V_i$.
\end{enumerate}
\end{lemma}

\begin{proof}
(1) Since $M\subseteq T\subseteq \Lambda$, then $M_i\subseteq T_i\subseteq P_i$ and hence $S_n=\soc M_i\subseteq\soc T_i\subseteq\soc P_i=S_n$.

(2) is clear. (3) is immediate from (1).

To prove (4), it suffices to show that $T_i$ can be recovered from
$\overline{T_i}$. If $\overline{T_i}=0$, then $T_i=M_i$.
Thus we can assume $\overline{T_i}\neq0$. Then $\overline{P_i}$ is an injective hull of $\overline{T_i}$ as a $\overline{\Lambda}$-module, and the natural epimorphism
$\pi:P_i\to\overline{P_i}$ is a projective cover of $\overline{P_i}$ as a
$\Lambda$-module. Since $T_i=\pi^{-1}(\overline{T_i})$ holds, the assertion follows.
\end{proof}

The following results on minimal left approximations are also needed
to construct support $\tau$-tilting $\Lambda$-modules.

\begin{lemma}\label{4.4}
Let $T\in\langle I_1,\ldots,I_{n-1}\rangle$.
\begin{enumerate}[\rm(1)]
\item The minimal left $\add(\bigoplus_{j=1}^{i-1}T_j)$-approximation of $T_i$ is
given by $f_i:T_i\to T_{i-1}$, which is the left multiplication of
the arrow $a_{i-1}:i-1\to i$ in the quiver of $\Lambda$. In this
case, $f_i(M_{i})=M_{i-1}$.
\item The minimal left $\add(\bigoplus_{j=i+1}^n T_j)$-approximation of $T_i$ is
given by $g_i:T_i\to T_{i+1}$, which is the left multiplication of
the arrow $b_{i+1}:i+1\to i$ in the quiver of $\Lambda$. This is a
monomorphism.
\end{enumerate}
\end{lemma}

\begin{proof}

(1) Since the left multiplication gives an isomorphism
$\Lambda\cong\End_\Lambda(T)$, we have an equivalence
$\Hom_\Lambda(T,-):\add T\cong\add\Lambda$. The minimal left
$\add(\bigoplus_{j=1}^{i-1}e_j\Lambda)$-approximation of
$e_i\Lambda$ is $e_i\Lambda\to e_{i-1}\Lambda$, which is given by
the left multiplication of $a_{i-1}$. Thus the former assertion
follows. The latter assertion follows from
$f_i(M_{i})=a_{i-1}M_{i}=M_{i-1}$.

(2) One can prove the first assertion similarly to (1). Since the
left multiplication of $b_{i+1}$ gives a monomorphism $P_{i}\to
P_{i+1}$, its restriction $g_i$ is also a monomorphism.
\end{proof}

Let $T\in\langle I_1,\ldots,I_{n-1}\rangle$ be a tilting $\Lambda$-module.
%Recall that $\mu_i(T)$ is the mutation of $T$ at the direct summand $T_i=e_iT$
%for $1\le i\le n-1$ (Corollary \ref{3.12}).
For $0\le i\le n$, we define
\[\mu_{[i+1,n]}(T):=\mu_{i+1}\mu_{i+2}\cdots\mu_n(T)\in\sttilt\Lambda\]
as the successive mutation at the direct summands $T_n,T_{n-1},\ldots,T_{i+1}$
(Definition \ref{mutation}), where $\mu_{[i+1, n]}(T):=T$ for $i=n$.
The following result plays a crucial role.

\begin{proposition}\label{4.5}
Let $T\in\langle I_1,\ldots,I_{n-1}\rangle$.
For $0\le i\le n$, we have
\begin{enumerate}[\rm(1)]
\item $\mu_{[i+1,n]}(T)=\bigoplus_{j=1}^iT_j\oplus\bigoplus_{j=i}^{n-1}\overline{T_j}$.
\item $T>\mu_n(T)>\mu_{[n-1,n]}(T)>\cdots>\mu_{[1,n]}(T)$.
\item Let $i\le j\le n-1$. Then $\overline{T_j}=0$ if and only if $S_{n-j}$ is not a composition factor of $\mu_{[i+1,n]}(T)$.
\item $(\mu_{[i+1,n]}(T),P)$ is a support $\tau$-tilting pair for $P:=\bigoplus_{i\le j\le n-1,\overline{T_j}=0}P_{n-j}$.
\end{enumerate}
\end{proposition}

\begin{proof} (1) We prove the assertion by descending induction on $i$.
It is clear for $i=n$.

Now we assume that $\mu_{[i+1,n]}(T)$ is $\bigoplus_{j=1}^i
T_j\oplus\bigoplus_{j=i}^{n-1} \overline{T_j}$. In the following we
calculate $\mu_{[i,n]}(T)$ by applying Theorem \ref{2.20}

Firstly, we show that $T_i\notin
\Fac(\bigoplus_{j=1}^{i-1}T_j\oplus\bigoplus_{j=i}^{n-1}
\overline{T_j})$. By Lemma \ref{4.3}(3), we have
$\Hom_\Lambda(\overline{T_j},T_i)=0$. Thus we only have to show
$T_i\notin\Fac(\bigoplus_{j=1}^{i-1}T_j)$. Otherwise, since $TM=M$
holds as ideals of $\Lambda$, we have
$e_iM=T_iM\in\Fac(\bigoplus_{j=1}^{i-1}T_jM)=\Fac(\bigoplus_{j=1}^{i-1}e_jM)$.
This is impossible by the explicit form of $M$. Thus the assertion
follows.

Next, by Lemma \ref{4.4}(1) and the fact that the natural epimorphism $\pi_i:T_i\rightarrow\overline{T_i}$ is a left
$(\mod\overline{\Lambda})$-approximation of $T_i$, a left
$\add(\bigoplus_{j=1}^{i-1}T_j\oplus\bigoplus_{j=i}^{n-1}\overline{T_j})$-approximation
of $T_i$ is given by $f:={f_i\choose\pi_i}:T_i\to
T_{i-1}\oplus\overline{T_i}$.

Finally, we have a commutative diagram of exact sequences
\[\xymatrix{
0\ar[r]&M_{i}\ar[r]\ar@{=}[d]&T_i\ar[r]^{\pi_i}\ar[d]^{f_i}&\overline{T_i}\ar[r]\ar[d]&0\\
&M_{i}\ar[r]&T_{i-1}\ar[r]&{\rm Coker} f\ar[r]&0, }\] we have ${\rm
Coker} f=T_{i-1}/f_i(M_{i})=\overline{T_{i-1}}$ by Lemma
\ref{4.4}(1). This is indecomposable by Lemma \ref{4.3}(2), and we
have $\mu_{[i,n]}(T)=
\bigoplus_{j=1}^{i-1}T_j\oplus\bigoplus_{j=i-1}^{n-1}
\overline{T_j}$ by Theorem \ref{2.20}. Thus the assertion follows.

(2) By the proof of (1) we get $\mu_{[i,n]}(T)$ is a left mutation
of $\mu_{[i+1,n]}(T)$, and hence the assertion holds.

(3) Notice that the $\Lambda$-module $\overline{P_j}$ has the socle $S_{n-j}$.
Since $\overline{T_j}$ is a submodule of $\overline{P_j}$, the ``if'' part follows.
Conversely, assume $\overline{T_j}=0$. Since $\overline{T}$ is a two-sided
ideal of the selfinjective $K$-algebra $\overline{\Lambda}$,
our assumption $\overline{T_j}=0$ implies that the $\overline{\Lambda}$-module
$\overline{T}$ does not have $S_{n-j}$ as a composition factor.
Since $M_k$ with $1\le k\le j$ does not have $S_{n-j}$ as a composition factor,
so does $\mu_{[i+1,n]}(T)$.

(4) This is immediate from (3).
\end{proof}

Now we give an example of calculation given in Proposition
\ref{4.5}.

\begin{example}\label{4.6}
Let $\Lambda$ be the Auslander algebra of $K[x]/(x^4)$. Taking the
trivial tilting module $\Lambda$, then $\mu_4(\Lambda)$,
$\mu_3\mu_4(\Lambda)$, $\mu_2\mu_3\mu_4(\Lambda)$ and
$\mu_1\mu_2\mu_3\mu_4(\Lambda)$ are given as follows.
\[\begin{xy}
(0,0)*+{ \left[\begin{smallmatrix} 1\\ &2\\
&&3\\&&&4\\ \end{smallmatrix}\middle| \begin{smallmatrix} &2\\ 1&&3\\
&2&&4\\&&3\\&&&4\end{smallmatrix}\middle| \begin{smallmatrix} &&3\\ &2&&4\\
1&&3\\&2&&4\\&&3\\&&&4\\ \end{smallmatrix}\middle| \begin{smallmatrix}&&&4 \\&&3\\ &2&&4\\
1&&3\\&2&&4\\&&3\\&&&4\\ \end{smallmatrix}\right]}="1",
(50,0)*+{ \left[\begin{smallmatrix} 1\\ &2\\
&&3\\&&&4\\ \end{smallmatrix}\middle| \begin{smallmatrix} &2\\ 1&&3\\
&2&&4\\&&3\\&&&4\end{smallmatrix}\middle| \begin{smallmatrix} &&3\\ &2&&4\\
1&&3\\&2&&4\\&&3\\&&&4\\ \end{smallmatrix}\middle|
\begin{smallmatrix}&&3 \\&2\\1\end{smallmatrix}\right]}="2",
(100,0)*+{ \left[\begin{smallmatrix} 1\\ &2\\
&&3\\&&&4\\ \end{smallmatrix}\middle| \begin{smallmatrix} &2\\ 1&&3\\
&2&&4\\&&3\\&&&4\end{smallmatrix}\middle|\begin{smallmatrix} &2\\ 1&&3\\
&2\end{smallmatrix}\middle| \begin{smallmatrix}&&3
\\&2\\1\end{smallmatrix}\right]}="3",
(100,-20)*+{ \left[\begin{smallmatrix} 1\\ &2\\
&&3\\&&&4\\ \end{smallmatrix}\middle| \begin{smallmatrix} 1\\ &2\\
&&3\end{smallmatrix}\middle|\begin{smallmatrix} &2\\ 1&&3\\
&2\end{smallmatrix}\middle| \begin{smallmatrix}&&3
\\&2\\1\end{smallmatrix}\right]}="4",
(50,-20)*+{\left[ \begin{smallmatrix} 1\\ &2\\
&&3\end{smallmatrix}\middle|\begin{smallmatrix} &2\\ 1&&3\\
&2\end{smallmatrix}\middle| \begin{smallmatrix}&&3
\\&2\\1\end{smallmatrix}\right]}="5",
\ar^{\mu_4}"1";"2",\ar^{\mu_3}"2";"3",\ar^{\mu_2}"3";"4",\ar_{\mu_1}"4";"5",
\end{xy}\]
\end{example}

For $0\leq i\leq n$, we denote by $\mu_{[i+1,n]}(\tilt\Lambda)$ the set of isomorphism
classes of support $\tau$-tilting $\Lambda$-modules consisting of
$\mu_{[i+1,n]}(T)$ for any $T\in\tilt\Lambda$.
Then we have the following proposition.

\begin{lemma}\label{4.7}
\begin{enumerate}[\rm(1)]
 \item  For any $0\leq i\leq n$, there is a bijection
$\tilt\Lambda\rightarrow \mu_{[i+1,n]}(\tilt\Lambda)$, which maps $T$ to $\mu_{[i+1,n]}(T)$.

\item We have $\mu_{[i+1,n]}(\tilt\Lambda)\cap
\mu_{[j+1,n]}(\tilt\Lambda)=\emptyset$ for any $0\leq i\neq j\leq
n$.

\end{enumerate}
\end{lemma}

\begin{proof} (1) This is clear since each $\mu_j:\sttilt\Lambda\to\sttilt\Lambda$ is a bijection.

(2) By Proposition \ref{4.5} and Lemma \ref{4.3}(1)(3), the first
$i$ direct summands of $\mu_{[i+1,n]}(T)$ have a composition factor
$S_n$, and the other summands do not have a composition factor
$S_n$. Thus the assertion follows.
%\new{(3) It suffices to show that $T\geq U=\mu_k(T)$ implies
%$\mu_{[i+1,n]}(T)\geq\mu_{[i+1,n]}(U)$. Assume that
%$T=\bigoplus_{j=1}^n T_j$ and $U=\mu_k(T)=\bigoplus_{j=1}^{k-1}
%T_j\oplus T'_k\oplus\bigoplus_{j=k+1}^{n}T_j$. Applying Proposition
%\ref{4.5} to $T$ and $U$ we get $\mu_{[i+1,n]}(T)=\bigoplus_{j=1}^i
%T_i\oplus\bigoplus_{j=i}^{n-1}\overline{T_j}$, and we describe
%$\mu_{[i+1,n]}(U)$ in the following three cases:
%(a) If $i\geq k+1$, then
%$\mu_{[i+1,n]}(U)=\bigoplus_{j=1}^{k-1}T_j\oplus
%T'_k\oplus\bigoplus_{j=k+1}^{i}T_j\oplus\bigoplus_{j=i}^{n-1}\overline{T_j}$.
%(b) If $i\leq k-1$, then
%$\mu_{[i+1,n]}(U)=\bigoplus_{j=1}^{i}T_j\oplus
%\bigoplus_{j=i}^{k-1}\overline{T_j}\oplus
%\overline{T'_k}\bigoplus_{j=k+1}^{n-1}\overline{T_j}$. }
\end{proof}

Let $U=\mu_{[i+1,n]}(T)\in\sttilt\Lambda$ with $T\in\langle I_1,\ldots,I_{n-1}\rangle$ and $0\le i\le n$, given in Proposition \ref{4.5}(1).
For each $1\le k\le n$, we define $\mu_k(U)$ by
\begin{equation}\label{define mu_i}
\mu_k(U)=\left\{\begin{array}{ll}
\mbox{the mutation of $U$ at $T_k$}&\mbox{if $1\le k\le i$,}\\
\mbox{the mutation of $U$ at $\overline{T_{k-1}}$}&\mbox{if $i+1\le k\le n$ and $\overline{T_{k-1}}\neq0$,}\\
\mbox{the mutation of $U$ at $P_{n-k+1}$}&\mbox{if $i+1\le k\le n$ and $\overline{T_{k-1}}=0$,}\end{array}\right.
\end{equation}
where the third case is well-defined by Proposition \ref{4.5}(4).
We have the following relations of mutation in $\sttilt\Lambda$
corresponding to Lemma \ref{4.1}(2).

\begin{proposition}\label{4.9}
Let $T\in\langle I_1,\ldots,I_{n-1}\rangle$, $0\le i\le n$ and
$U:=\mu_{[i+1,n]}(T)$.
\begin{enumerate}[\rm(1)]
\item For any $1\leq k\leq i-1$, we have
$\mu_k(U)=\mu_{[i+1,n]}(\mu_k(T))$. Moreover, $T>\mu_k(T)$ if and
only if $U>\mu_k(U)$.
%This is a left (resp. right) mutation of $U$ if and only if $\mu_k(T)$ is a left (resp. right) mutation of $T$.
\item For any $i+2\leq k\leq n$, we have
$\mu_k(U)=\mu_{[i+1,n]}(\mu_{k-1}(T))$. Moreover, $T>\mu_{k-1}(T)$
if and only if $U>\mu_k(U)$.
\item We have \[\mu_k\mu_{[i+1,n]}(T)=
   \begin{cases}
   \mu_{[i+1,n]}\mu_k(T) &\mbox{$k\leq i-1$,}\\
   \mu_{[i,n]}(T)&\mbox{$k=i$,}\\
   \mu_{[i+2,n]}(T)&\mbox{$k=i+1$,}\\
   \mu_{[i+1,n]}\mu_{k-1}(T) &\mbox{$k\geq i+2$.}\\
   \end{cases}\]
\end{enumerate}
\end{proposition}

\begin{proof}
By Proposition \ref{4.5}, we have
$U=\bigoplus_{j=1}^{i}T_j\oplus\bigoplus_{j=i}^{n-1}\overline{T_j}$.

(1) Let $V:=\mu_k(T)=\bigoplus_{j=1}^{k-1} T_j\oplus
V_k\oplus\bigoplus_{j=k+1}^{n}T_j$. Then $V$ is a tilting
$\Lambda$-module with $V_k\not\cong T_k$, and applying Proposition
\ref{4.5} to $V$, we have
$\mu_{[i+1,n]}(V)=\bigoplus_{j=1}^{k-1}T_j\oplus
V_k\oplus\bigoplus_{j=k+1}^{i}T_j\oplus\bigoplus_{j=i}^{n-1}\overline{T_j}$.
Since $U$ and $\mu_{[i+1,n]}(V)$ have the same indecomposable direct
summands except the $k$-th one, we have $\mu_k(U)=\mu_{[i+1,n]}(V)$
as desired.

To prove the latter one, it suffices to show that $T>\mu_k(T)$
implies $U>\mu_k(U)$. The condition $T>\mu_k(T)$ is equivalent to
$T_k\notin\Fac(T/T_k)$. Since $U/U_k$ belongs to $\Fac(T/T_k)$ by
the explicit form in Proposition \ref{4.5}, we have
$U_k=T_k\notin\Fac(U/U_k)$. Therefore $U>\mu_k(U)$.

(2) Let $V:=\mu_{k-1}(T)=\bigoplus_{j=1}^{k-2} T_j\oplus
V_{k-1}\oplus\bigoplus_{j=k}^{n}T_j$. Then $V$ is a tilting
$\Lambda$-module with $V_{k-1}\not\cong T_{k-1}$, and applying
Proposition \ref{4.5} to $V$, we have
$\mu_{[i+1,n]}(V)=\bigoplus_{j=1}^{i}T_j\oplus
\bigoplus_{j=i}^{k-2}\overline{T_j}\oplus\overline{V_{k-1}}\oplus\bigoplus_{j=k}^{n-1}\overline{T_j}$.
Since $\overline{V_{k-1}}\not\cong\overline{T_{k-1}}$ holds by
Lemma \ref{4.3}(4), $U$ and $\mu_{[i+1,n]}(V)$ have the same
indecomposable direct summands except the $k$-th one. Thus we have
$\mu_k(U)=\mu_{[i+1,n]}(V)$ as desired.

To show the latter one, it suffices to show that $T<\mu_{k-1}(T)$
implies $U<\mu_k(U)$. The condition $T<\mu_{k-1}(T)$ is equivalent
to $T_{k-1}\in\Fac(T/T_{k-1})$. Since
$\overline{T}/\overline{T_{k-1}}$ belongs to $\Fac(U/U_k)$ by the
explicit form in Proposition \ref{4.5}, we have $U_k=\overline{T_{k-1}}\in
\Fac(\overline{T}/\overline{T_{k-1}})\subseteq\Fac(U/U_k)$. Therefore
$U<\mu_k(U)$.

(3) Immediate from (1) and (2).
\end{proof}

Immediately we have the following complete classification of support
$\tau$-tilting $\Lambda$-modules and indecomposable $\tau$-rigid
$\Lambda$-modules.

\begin{theorem}\label{4.14}
\begin{enumerate}[\rm(1)]
\item We have $\sttilt\Lambda=\bigsqcup_{i=0}^{n}\mu_{[i+1,n]}(\tilt\Lambda)$.
In particular, $\#\sttilt\Lambda=(n+1)!$,
and the mutation $\mu_k$ for each $1\le k\le n$
is well-defined on $\sttilt\Lambda$ by \eqref{define mu_i}.
\item Any support $\tau$-tilting $\Lambda$-module has a form
$T_1\oplus\cdots\oplus
T_i\oplus\overline{T_i}\oplus\cdots\oplus\overline{T_{n-1}}$ for
some $0\le i\le n$ and $T\in\langle I_1,\ldots,I_{n-1}\rangle$ with
$T_j:=e_jT$ for $1\le j\le n$. Moreover such $i$ and $T$ are
uniquely determined.
\item Any indecomposable $\tau$-rigid module has a form $T_i=e_iT$ or $\overline{T_i}$
for some $T\in\langle I_1,\ldots,I_{n-1}\rangle$ and $1\le i\le n$.
\item The statements (1) and (2) hold for $\Lambda^{\op}$-modules.
\end{enumerate}
\end{theorem}

\begin{proof}
(1) By Lemma \ref{4.7},
$\bigcup_{i=0}^{n}\mu_{[i+1,n]}(\tilt\Lambda)$ is a disjoint union
and contains precisely $(n+1)!$ elements.
By Proposition \ref{4.9}(3), $\bigsqcup_{i=0}^{n}\mu_{[i+1,n]}(\tilt\Lambda)$ is
closed under mutation. This is a finite connected component of ${\rm
H}(\sttilt\Lambda)$. By Proposition
\ref{2.22}, we have
$\sttilt\Lambda=\bigsqcup_{i=0}^{n}\mu_{[i+1,n]}(\tilt\Lambda)$.

(2) is clear by (1) and Proposition \ref{4.5}. (3) is a straight
result of (2) and Lemma \ref{2.18}.
\end{proof}

The following lemma is also needed.

\begin{lemma}\label{4.b}
 Let $U\in\sttilt\Lambda$ and $1\le j,k\le n$.
 \begin{enumerate}[\rm(1)]
 \item $\mu_j\mu_j(U)=U$.
 \item If $|j-k|\ge2$, then $\mu_j\mu_k(U)= \mu_k\mu_j(U)$.
 \item If $|j-k|=1$, then $\mu_j\mu_{k}\mu_{j}(U)=\mu_{k}\mu_{j}\mu_{k}(U)$.
\end{enumerate}
\end{lemma}

\begin{proof}
(1) is clear from the definition of mutation.

By Theorem \ref{4.14}(1), we can assume that $U=\mu_{[i+1,n]}(T)$
for some $0\leq i\leq n$ and $T\in\langle
I_1,\ldots,I_{n-1}\rangle$. In the following we use Proposition
\ref{4.9}(3) and Proposition \ref{3.15} frequently.

(2) Without loss of generality, we assume $k<j$. We divide the proof
into seven cases.

(a) If $k<j\leq i-1$, then
$\mu_j\mu_k(U)=\mu_j\mu_k\mu_{[i+1,n]}(T)=\mu_j\mu_{[i+1,n]}\mu_k(T)=\mu_{[i+1,n]}\mu_j\mu_k(T)$
$=\mu_{[i+1,n]}\mu_k\mu_j(T)=\mu_k\mu_j\mu_{[i+1,n]}(T)=\mu_k\mu_j(U)$.

(b) If $i+2\leq k<j$, then the proof is very similar to (a).

(c) If $k\leq i-1<i+2\leq j$, then
$\mu_j\mu_k(U)=\mu_j\mu_k\mu_{[i+1,n]}(T)=\mu_{j}\mu_{[i+1,n]}\mu_k(T)=\mu_{[i+1,n]}\mu_{j-1}\mu_k(T)=\mu_{[i+1,n]}\mu_k\mu_{j-1}(T)=\mu_k\mu_{[i+1,n]}\mu_{j-1}(T)=\mu_k\mu_{j}\mu_{[i+1,n]}(T)=\mu_k\mu_j(U)$.

(d) The case $k=i<i+2\leq j$, then
$\mu_j\mu_k(U)=\mu_j\mu_k\mu_{[i+1,n]}(T)=\mu_j\mu_{[i,n]}(T)=\mu_{[i,n]}\mu_{j-1}(T)
=\mu_k\mu_{[i+1,n]}\mu_{j-1}(T)=\mu_k\mu_j\mu_{[i+1,n]}(T)=\mu_k\mu_j(U)$.

(e) If $k\leq i-2<i=j$, then the proof is very similar to (d).

(f) If $k\leq i-1<i+1=j$, then
$\mu_j\mu_k(U)=\mu_j\mu_k\mu_{[i+1,n]}(T)=\mu_{i+1}\mu_{[i+1,n]}\mu_k(T)=\mu_{[i+2,n]}\mu_k(T)=\mu_k\mu_{[i+2,n]}(T)=\mu_k\mu_j\mu_{[i+1,n]}(T)=\mu_k\mu_j(U)$.

(g) If $k=i+1<i+3\leq j$, then the proof is very similar to (d).

(3)  Without loss of generality, we assume $k=j+1$. We also divide
the proof into five cases.

(a) If $j\leq i-2$, then
$\mu_j\mu_{k}\mu_{j}(U)=\mu_j\mu_{k}\mu_{j}\mu_{[i+1,n]}(T)=\mu_{[i+1,n]}\mu_j\mu_{k}\mu_{j}(T)
=\mu_{[i+1,n]}\mu_{k}\mu_{j}\mu_{k}(T)=\mu_{k}\mu_{j}\mu_{k}\mu_{[i+1,n]}(T)=\mu_{k}\mu_{j}\mu_{k}(U)$.

(b) If $j\geq i+2$, then the proof is very similar to (a).

(c) If $j=i-1$, then
$\mu_{i-1}\mu_{i}\mu_{i-1}(U)=\mu_{i-1}\mu_{i}\mu_{i-1}\mu_{[i+1,n]}(T)=\mu_{i-1}\mu_{i}\mu_{[i+1,n]}\mu_{i-1}(T)=\mu_{[i-1,n]}\mu_{i-1}(T)=\mu_i\mu_{[i-1,n]}(T)=\mu_{i}\mu_{i-1}\mu_{i}\mu_{[i+1,n]}(T)=\mu_{i}\mu_{i-1}\mu_{i}(U)$.

(d) If $j=i$ or $j=i+1$, then the proof is very similar to (c).
\end{proof}

Now we are in a position to state one of the main results of this
section.

\begin{theorem}\label{4.13}
Let $\Lambda$ be the Auslander algebra of $K[x]/(x^n)$. Then
\begin{enumerate}[\rm(1)]
\item There exists a bijection $I:\mathfrak{S}_{n+1}\cong\sttilt\Lambda$
which maps $w$ to $I(w)=\mu_{i_1}\mu_{i_2}\cdots\mu_{i_l}(\Lambda)$,
where $w=s_{i_{1}}s_{i_{2}}\cdots s_{i_{l}}$ is an arbitrary (not
necessarily reduced) expression.

\item The statement (1) holds for $\Lambda^{\op}$-modules.
\end{enumerate}
\end{theorem}

\begin{proof}
(1) Proposition \ref{4.b} and the same argument as in the proof of
Theorem \ref{3.17} show that the map $I$ is well-defined. By
Theorem \ref{4.14}, we have
$\#\sttilt\Lambda=(n+1)!=\#\mathfrak{S}_{n+1}$. Thus we only have to
show $I$ is surjective.

By Theorem \ref{4.14}, any $U\in\sttilt\Lambda$ is written as
$\mu_{[i+1,n]}(T)$ for some $T\in\tilt\Lambda$ and $0\le i\le n$.
By Corollary \ref{3.a} , there exists $w\in\mathfrak{S}_n$ such that
$T=I(w)$. Then we have $I(s_{i+1}\cdots s_nw)=\mu_{[i+1,n]}(T)=U$.
Thus the assertion follows.

(2) We only need to replace $\Lambda$-modules with
$\Lambda^{\op}$-modules in the proof.
 \end{proof}

Our second goal in this section is to show that the map $I$ in
Theorem \ref{4.13} is an anti-isomorphism of posets. For this aim,
we need the following result.

\begin{proposition}\label{4.15}
 For $w \in \mathfrak{S}_{n+1}$ and $1 \leq j
\leq n$, $l (s_{j}w) > l (w)$ if and only if $I(s_{j}w) < I(w)$.
\end{proposition}

\begin{proof}
It suffices to show that $l(s_{j}w)>l(w)$ implies that
$I(s_{j}w)<I(w)$ by replacing $s_{j}w$ with $w$ if necessary.  Write
$w = s_{i+1}\cdots s_{n}v$ with $0\le i\le n$ and
$v\in\mathfrak{S}_n$. Then $l(w)=n-i+l(v)$ and $l(s_jw)=n-i+l(v)+1$
hold by our assumption.
 We prove the assertion by comparing $i$ with
$j$.

(a) Assume $j\leq i-1$. By Proposition \ref{4.9}(3), we have
 $I(s_jw)=\mu_j\mu_{[i+1,n]}(I(v))=\mu_{[i+1,n]}\mu_{j}(I(v))=\mu_{[i+1,n]}(I(s_jv))$.
 Since $s_jw=s_{i+1}\cdots s_ns_jv$ holds, we have $n-i+l(v)+1=l(s_jw)\le n-i+l(s_jv)$
 and hence $l(v)+1=l(s_jv)$. Then by Theorem \ref{3.17} one has $I(s_jv)<I(v)$,
 which implies by Proposition \ref{4.9}(1) that $I(s_jw)=\mu_{[i+1,n]}(I(s_jv))<\mu_{[i+1,n]}(I(v))=I(w)$.

 (b) Assume $j\geq i+2$. We have
 $I(s_jw)=\mu_j\mu_{[i+1,n]}(I(v))=\mu_{[i+1,n]}\mu_{j-1}(I(v))=\mu_{[i+1,n]}(I(s_{j-1}v))$ by Proposition \ref{4.9}(3).
 Since $s_jw=s_{i+1}\cdots s_ns_{j-1}v$ holds by Lemma \ref{4.1}(2),
 we have $n-i+l(v)+1=l(s_jw)\le n-i+l(s_{j-1}v)$ and hence $l(v)+1=l(s_{j-1}v)$.
 Then by Theorem \ref{3.17} one has $I(s_{j-1}v)<I(v)$,
 which implies by Proposition \ref{4.9}(2) that $I(s_jw)=\mu_{[i+1,n]}(I(s_{j-1}v))<\mu_{[i+1,n]}(I(v))=I(w)$.

 (c) Assume $j=i$. By Proposition \ref{4.9}(3), we have
 $I(s_jw)=\mu_i\mu_{[i+1,n]}(I(v))=\mu_{[i,n]}(I(v))<\mu_{[i+1,n]}(I(v))=I(w)$ by Proposition \ref{4.5}(2).

(d) The case $j=i+1$ does not occur. In fact $s_jw=s_{i+2}\cdots
s_nv$ implies $l(s_jw)=l(w)-1$, a contradiction.
\end{proof}

Now we are ready to show the main result on the anti-isomorphisms of
posets.

\begin{theorem}\label{4.16}
Let $\Lambda$ and $I$ be as in Theorem \ref{4.13}. Then
$I:\mathfrak{S}_{n+1}\to\sttilt\Lambda$ is an anti-isomorphism of
posets with respect to the left order on $\mathfrak{S}_{n+1}$ and the generation order on $\sttilt\Lambda$, that is, $w_1\leq w_2$ in $\mathfrak{S}_{n+1}$ if and only
if $I(w_1)\geq I(w_2)$ in $\sttilt\Lambda$.
\end{theorem}

\begin{proof}
The proof is very similar to the proof of Theorem  \ref{3.17}(3), we
need to use Proposition \ref{4.15} instead of Proposition
\ref{two actions same}(3).
\end{proof}

To compare with the Hasse quiver of support $\tau$-tilting
$\Lambda$-modules, we give the Hasse quiver of the left order on the
symmetric group $\mathfrak{S}_{n}$ for $n=4$.
\begin{example}\label{4.17}
%We describe the Hasse quiver of the left order on
%$\mathfrak{S}_{4}$.
The Hasse quiver of the left order on
 $\mathfrak{S}_{4}$ is the \emph{opposite} of the following quiver:
$$
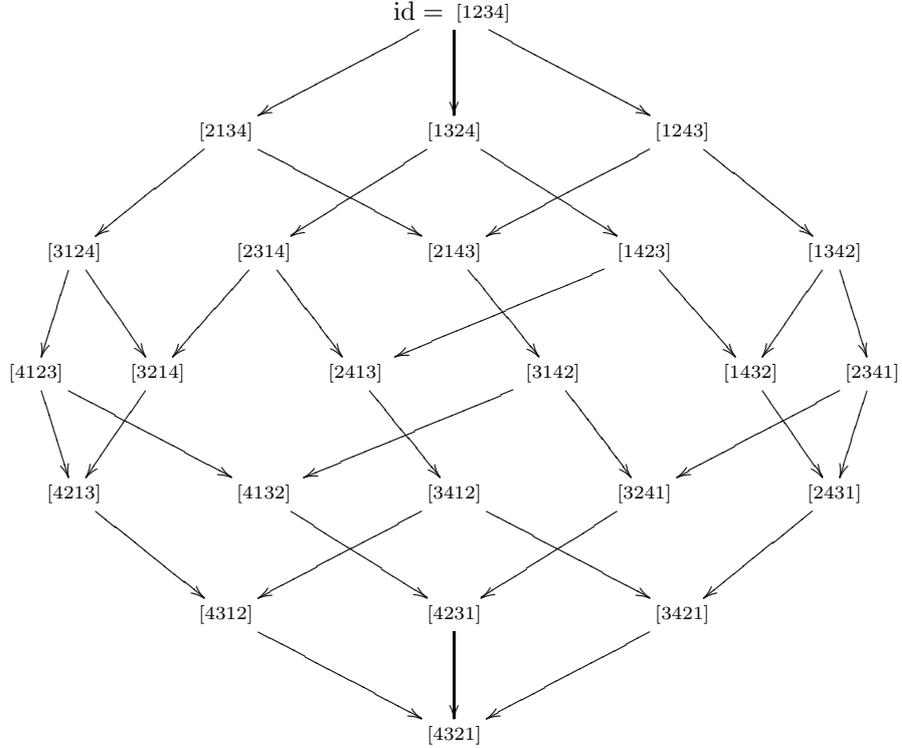
\begin{xy}
0;<3pt,0pt>:<0pt,2pt>::
(0,0)*+{{\rm id}=\begin{smallmatrix}[1234]\end{smallmatrix}}="1",
(-30,-16)*+{\begin{smallmatrix}[2134]\end{smallmatrix}}="2",
(0,-16)*+{\begin{smallmatrix}[1324]\end{smallmatrix}}="3",
(30,-16)*+{
\begin{smallmatrix}[1243]\end{smallmatrix}}="4",
(-50,-32)*+{
\begin{smallmatrix}[3124]\end{smallmatrix}}="5",
(-25,-32)*+{
\begin{smallmatrix}[2314]\end{smallmatrix}}="6",
(0,-32)*+{
\begin{smallmatrix}[2143]\end{smallmatrix}}="7",
(25,-32)*+{
\begin{smallmatrix}[1423]\end{smallmatrix}}="8",
(50,-32)*+{
\begin{smallmatrix}[1342]\end{smallmatrix}}="9",
(-55,-48)*+{
\begin{smallmatrix}[4123]\end{smallmatrix}}="10",
(-39,-48)*+{
\begin{smallmatrix}[3214]\end{smallmatrix}}="11",
(-13,-48)*+{
\begin{smallmatrix}[2413]\end{smallmatrix}}="12",
 (13,-48)*+{
\begin{smallmatrix}[3142]\end{smallmatrix}}="13",
(39,-48)*+{
\begin{smallmatrix}[1432]\end{smallmatrix}}="14",
(55,-48)*+{
\begin{smallmatrix}[2341]\end{smallmatrix}}="15",
(-50,-64)*+{
\begin{smallmatrix}[4213]\end{smallmatrix}}="16",
(-25,-64)*+{
\begin{smallmatrix}[4132]\end{smallmatrix}}="17",
(0,-64)*+{
\begin{smallmatrix}[3412]\end{smallmatrix}}="18",
(25,-64)*+{
\begin{smallmatrix}[3241]\end{smallmatrix}}="19",
(50,-64)*+{
\begin{smallmatrix}[2431]\end{smallmatrix}}="20",
(-30,-80)*+{
\begin{smallmatrix}[4312]\end{smallmatrix}}="21",
(0,-80)*+{
\begin{smallmatrix}[4231]\end{smallmatrix}}="22",
(30,-80)*+{
\begin{smallmatrix}[3421]\end{smallmatrix}}="23",
(0,-96)*+{
\begin{smallmatrix}[4321]\end{smallmatrix}}="24",
 \ar"1";"2",
\ar"1";"3", \ar"1";"4", \ar"2";"5", \ar"2";"7", \ar"3";"6",
\ar"3";"8", \ar"4";"7", \ar"4";"9", \ar"5";"10", \ar"5";"11",
\ar"6";"11", \ar"6";"12", \ar"7";"13", \ar"8";"12", \ar"8";"14",
\ar"9";"14", \ar"9";"15", \ar"10";"16", \ar"10";"17", \ar"11";"16",
\ar"13";"17", \ar"13";"19", \ar"12";"18", \ar"15";"19",
\ar"14";"20", \ar"15";"20", \ar"16";"21", \ar"17";"22",
\ar"18";"21", \ar"18";"23", \ar"19";"22", \ar"20";"23",
\ar"21";"24", \ar"22";"24", \ar"23";"24",
\end{xy}$$
\end{example}

By Theorem \ref{4.18}, we give the Hasse quiver of support
$\tau$-tilting modules of the Auslander algebra of $K[x]/(x^n)$ for
$n=2,3$.

\begin{example}\label{4.18}
Denote by $\Lambda_i$ the Auslander algebra of $K[x]/(x^i)$ for $i=2,3$. Then
\begin{enumerate}[\rm(1)]
 \item The Hasse quiver ${\rm H}(\sttilt\Lambda_2)$ is of the following form, where $\xrightarrow{i}$ shows $\mu_i$.
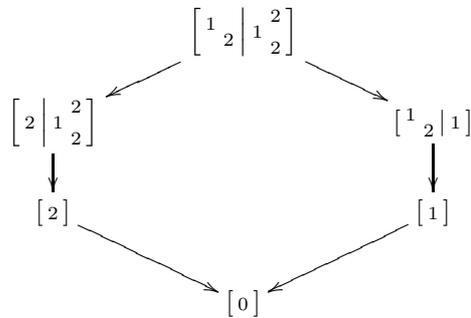
$$\begin{xy}
0;<3pt,0pt>:<0pt,2.2pt>::
(0,0)
*+{\left[\begin{smallmatrix} 1\\ &2\\ \end{smallmatrix}\middle|
\begin{smallmatrix} &2\\ 1\\&2\\ \end{smallmatrix}\right]}="1",
(-25,-4)
*+{\left[\begin{smallmatrix} 2\\ \end{smallmatrix}\middle|
\begin{smallmatrix} &2\\ 1\\&2\\ \end{smallmatrix}\right]}="2",
(25,-4)
*+{\left[\begin{smallmatrix}1\\&2\\ \end{smallmatrix}\middle|
\begin{smallmatrix} 1\\ \end{smallmatrix}\right]}="3",
(-25,-20)
*+{\left[\begin{smallmatrix} 2\\ \end{smallmatrix}\middle|\ \right]}="4",
(25,-20)
*+{\left[\ \middle|\begin{smallmatrix} 1\\ \end{smallmatrix}\right]}="5",
(0,-24)
*+{ \left[\ \middle|\
%\begin{smallmatrix} 0\\ \end{smallmatrix}
\right]}="6",
\ar"1";"2",|{1} \ar"1";"3",|{2} \ar"2";"4",|{2} \ar"3";"5",|{1} \ar"4";"6",|{1}
\ar"5";"6",|{2}
\end{xy}$$

\item The Hasse quiver ${\rm H}(\sttilt\Lambda_3)$ is of the following form, where $\xrightarrow{i}$ shows $\mu_i$.
$$\begin{xy}
0;<3.4pt,0pt>:<0pt,3.2pt>::
(0,0)*+{ \left[\begin{smallmatrix} 1\\ &2\\
&&3\end{smallmatrix}\middle| \begin{smallmatrix} &2\\ 1&&3\\
&2\\&&3\end{smallmatrix}\middle| \begin{smallmatrix} &&3\\ &2\\
1&&3\\&2\\&&3\end{smallmatrix}\right]}="1", (-30,-16)
*+{\left[\begin{smallmatrix} \\ 2\\
&3\end{smallmatrix}\middle| \begin{smallmatrix} &2\\ 1&&3\\
&2\\&&3\end{smallmatrix}\middle| \begin{smallmatrix} &&3\\ &2\\
1&&3\\&2\\&&3\end{smallmatrix}\right]}="2", (0,-16)
*+{\left[\begin{smallmatrix} 1\\ &2\\
&&3\end{smallmatrix}\middle| \begin{smallmatrix} 1&&3\\
&2\\&&3\end{smallmatrix}\middle| \begin{smallmatrix} &&3\\ &2\\
1&&3\\&2\\&&3\end{smallmatrix}\right]}="3", (30,-16)
*+{\left[\begin{smallmatrix} 1\\ &2\\
&&3\end{smallmatrix}\middle| \begin{smallmatrix} &2\\ 1&&3\\
&2\\&&3\end{smallmatrix}\middle| \begin{smallmatrix} &2\\
1\end{smallmatrix}\right]}="4", (-50,-32)
*+{\left[\begin{smallmatrix} 2\\ &3\end{smallmatrix}\middle|
\begin{smallmatrix} &3\\ 2\\&3\end{smallmatrix}\middle|
\begin{smallmatrix} &&3\\ &2\\1&&3\\&2\\&&3\end{smallmatrix}\right]}="5", (-25,-32)
*+{\left[\begin{smallmatrix} 3\end{smallmatrix}\middle|
\begin{smallmatrix} 1&&3\\ &2\\&&3\end{smallmatrix}\middle|
\begin{smallmatrix} &&3\\ &2\\
1&&3\\&2\\&&3\end{smallmatrix}\right]}="6", (0,-32)
*+{\left[\begin{smallmatrix} \\ 2\\
&3\end{smallmatrix}\middle| \begin{smallmatrix} &2\\ 1&&3\\
&2\\&&3\end{smallmatrix}\middle| \begin{smallmatrix} &2\\
1\end{smallmatrix}\right]}="7", (25,-32)
*+{\left[\begin{smallmatrix} 1\\ &2\\
&&3\end{smallmatrix}\middle| \begin{smallmatrix} 1&&3\\
&2\\&&3\end{smallmatrix}\middle| \begin{smallmatrix}
1\end{smallmatrix}\right]}="8", (50,-32)
*+{\left[\begin{smallmatrix} 1\\ &2\\
&&3\end{smallmatrix}\middle| \begin{smallmatrix} 1\\
&2\end{smallmatrix}\middle| \begin{smallmatrix} &2\\
1\end{smallmatrix}\right]}="9", (-55,-48)
*+{\left[\begin{smallmatrix} 2\\ &3\end{smallmatrix}\middle|
\begin{smallmatrix} &3\\ 2\\&3\end{smallmatrix}\middle|\ \right]}="10",
(-35,-48)
*+{\left[\begin{smallmatrix}
3\end{smallmatrix}\middle| \begin{smallmatrix} &3\\
2\\ &3\end{smallmatrix}\middle| \begin{smallmatrix} &&3\\
&2\\1&&3\\&2\\&&3\end{smallmatrix}\right]}="11", (-10,-48)
*+{\left[\begin{smallmatrix}
3\end{smallmatrix}\middle| \begin{smallmatrix} 1&&3\\
&2\\&&3\end{smallmatrix}\middle| \begin{smallmatrix}
1\end{smallmatrix}\right]}="12", (10,-48)
*+{\left[\begin{smallmatrix} \\ 2\\
&3\end{smallmatrix}\middle| \begin{smallmatrix}
2\end{smallmatrix}\middle| \begin{smallmatrix} &2\\
1\end{smallmatrix}\right]}="13", (35,-48)
*+{\left[\begin{smallmatrix} 1\\ &2\\
&&3\end{smallmatrix}\middle| \begin{smallmatrix} 1\\
&2\end{smallmatrix}\middle| \begin{smallmatrix}
1\end{smallmatrix}\right]}="14", (55,-48)
*+{\left[\ \middle|\begin{smallmatrix} 1\\ &2\end{smallmatrix}\middle|
\begin{smallmatrix} &2\\ 1\end{smallmatrix}\right]}="15",
(-50,-64)
*+{\left[\begin{smallmatrix}
3\end{smallmatrix}\middle| \begin{smallmatrix} &3\\
2\\&3\end{smallmatrix}\middle|\ \right]}="16", (-25,-64)
*+{\left[\begin{smallmatrix} 2\\
&3\end{smallmatrix}\middle| \begin{smallmatrix}
2\end{smallmatrix}\middle|\ \right]}="17", (-0,-64)
*+{\left[\begin{smallmatrix} 3\end{smallmatrix}\middle|\ \middle|
\begin{smallmatrix} 1\end{smallmatrix}\right]}="18", (25,-64)
*+{\left[\ \middle|\begin{smallmatrix}
2\end{smallmatrix}\middle| \begin{smallmatrix} &2\\
1\end{smallmatrix}\right]}="19", (50,-64)
*+{\left[\ \middle|\begin{smallmatrix} 1\\
&2\end{smallmatrix}\middle| \begin{smallmatrix}
1\end{smallmatrix}\right]}="20", (-30,-80)
*+{\left[\begin{smallmatrix}
3\end{smallmatrix}\middle|\ \middle|\ \right]}="21", (0,-80)
*+{\left[\ \middle|\begin{smallmatrix}
2\end{smallmatrix}\middle|\ \right]}="22", (30,-80)
*+{\left[\ \middle|\ \middle|\begin{smallmatrix}
1\end{smallmatrix}\right]}="23", (0,-96)
*+{\left[\ \middle|\ \middle|\ \right]}="24", \ar"1";"2",|{1}
\ar"1";"3",|{2} \ar"1";"4",|{3} \ar"2";"5",|{2} \ar"2";"7",|{3} \ar"3";"6",|{1}
\ar"3";"8",|{3} \ar"4";"7",|{1} \ar"4";"9",|{2} \ar"5";"10",|{3} \ar"5";"11",|{1}
\ar"6";"11",|{2} \ar"6";"12",|{3} \ar"7";"13",|{2} \ar"8";"12",|{1} \ar"8";"14",|{2}
\ar"9";"14",|{3} \ar"9";"15",|{1} \ar"10";"16",|{1} \ar"10";"17",|{2} \ar"11";"16",|{3}
\ar"12";"18",|{2} \ar"13";"17",|{3} \ar"13";"19",|{1} \ar"14";"20",|{1}
\ar"15";"19",|{2} \ar"15";"20",|{3} \ar"16";"21",|{2} \ar"17";"22",|{1}
\ar"18";"21",|{3} \ar"18";"23",|{1} \ar"19";"22",|{3} \ar"20";"23",|{2}
\ar"21";"24",|{1} \ar"22";"24",|{2} \ar"23";"24",|{3} \end{xy}$$
\end{enumerate}
\end{example}

\section{Connection with preprojective algebras of type $A_{n}$}

Let $\Lambda$ be the Auslander algebra of $K[x]/(x^n)$ and $\Gamma$ be
the preprojective algebra of Dynkin type $A_n$. Thus $\Gamma$ is
presented by the quiver
\[\xymatrix{
1\ar@<2pt>[r]^{a_1}&2\ar@<2pt>[r]^{a_2}\ar@<2pt>[l]^{b_2}&3\ar@<2pt>[r]^{a_3}\ar@<2pt>[l]^{b_3}&\cdots\ar@<2pt>[r]^{a_{n-2}}\ar@<2pt>[l]^{b_4}&n-1\ar@<2pt>[r]^{a_{n-1}}\ar@<2pt>[l]^{b_{n-1}}&n\ar@<2pt>[l]^{b_n}
}\] with relations $a_{1} b_{2}= 0$, $b_na_{n-1}=0$ and $a_{i} b_{i+1} =b_{i}
a_{i-1}$ for any $2 \leq i \leq n-1$. Thus we
have $\Gamma=\Lambda/L$ for the ideal $L$ of $\Lambda$ generated by
$b_na_{n-1}$. Then we have $L=\bigoplus_{i=1}^nL_i$ for $L_i:=e_iL$.
For example, $\Lambda$ and $L$ in the case $n=4$ is the following.
\[L=\left[\begin{smallmatrix} \ \\ &\ \\
&&\ \\&&&\ \\ \end{smallmatrix}\middle| \begin{smallmatrix} &\ \\ \ &&\ \\
&\ &&\ \\&&\ \\&&&4\end{smallmatrix}\middle| \begin{smallmatrix} &&\ \\ &\ &&\ \\
\ &&\ \\&\ &&4\\&&3\\&&&4\\ \end{smallmatrix}\middle| \begin{smallmatrix}&&&\  \\&&\ \\ &\ &&4\\
\ &&3\\&2&&4\\&&3\\&&&4\\ \end{smallmatrix}\right] \subseteq
\Lambda=\left[\begin{smallmatrix} 1\\ &2\\
&&3\\&&&4\\ \end{smallmatrix}\middle| \begin{smallmatrix} &2\\ 1&&3\\
&2&&4\\&&3\\&&&4\end{smallmatrix}\middle| \begin{smallmatrix} &&3\\ &2&&4\\
1&&3\\&2&&4\\&&3\\&&&4\\ \end{smallmatrix}\middle| \begin{smallmatrix}&&&4 \\&&3\\ &2&&4\\
1&&3\\&2&&4\\&&3\\&&&4\\ \end{smallmatrix}\right]
\]
Our aim in this section is to apply Theorems \ref{4.13} and
\ref{4.16} to $\Gamma$ and prove that the tensor functor
\[-\otimes_{\Lambda}\Gamma:\mod \Lambda\rightarrow\mod\Gamma\]
induces a bijection from $\sttilt\Lambda$ to $\sttilt\Gamma$. In
particular, we can get Mizuno's bijection from the symmetric group
$\mathfrak{S}_{n+1}$ to $\sttilt\Gamma$.

Let us start with the following general properties of support
$\tau$-tilting modules over an algebra $A$ and its factor algebra $B$.
%For the completeness, we give the proof.

\begin{proposition}{\cite{DIRRT}}\label{5.1}
Let $A$ be an algebra and let $B$ be a factor algebra of $A$.
\begin{enumerate}[\rm(1)]
\item If $T$ is a $\tau$-rigid $A$-module, then $T\otimes_{A}B$ is a $\tau$-rigid $B$-module.
\item If $T$ is a support $\tau$-tilting $A$-module, then $T\otimes_{A}B$ is a support $\tau$-tilting
$B$-module. Thus we have a map $-\otimes_AB:\sttilt A\to\sttilt B$,
which preserves the generation order.
\item The map in $(2)$ is surjective if $A$ is $\tau$-rigid finite.
\end{enumerate}
\end{proposition}

Note that $T\otimes_AB$ is not necessarily basic even if $T$ is basic $\tau$-rigid.

Recall that $M$ and $L$ are the ideals defined at the beginning of Sections 4 and 5 respectively, and $M_i=e_iM$ and $L_i=e_iL$ for $1\le i\le n$.
We need the following facts.

\begin{lemma}\label{4.2}
Let $T\in\langle I_1,\ldots,I_{n-1}\rangle$ and $T_i:=e_iT$ for
$1\le i\le n$. For any $1\le i\le n$, we have
\begin{enumerate}[\rm(1)]
\item $LM=L=ML$ and $T_iL=L_i$.
\item $T_i/L_{i}$ is indecomposable with a simple socle $S_{n-i+1}$.
\item Let $V\in\langle I_1,\ldots,I_{n-1}\rangle$.
If $T_i/L_i\cong V_i/L_i$, then $T_i\cong V_i$.
\end{enumerate}
\end{lemma}

\begin{proof}

(1) This is clear. (2) Since $M_i\subseteq T_i\subseteq P_i$, we have
$L_i=M_iL\subseteq T_iL\subseteq P_iL=L_i$. The socle of
$T_i/L_i\subseteq P_i/L_i$ is $S_{n-i+1}$. (3) One can prove in a
similar method with Lemma \ref{4.3}(4).
\end{proof}

Now we can state our main result of this section.

\begin{theorem}\label{5.4}
Let $\Lambda$ be the Auslander algebra of $K[x]/(x^n)$ and $\Gamma$
the preprojective algebra of Dynkin type $A_n$.
\begin{enumerate}[\rm(1)]
\item The map $-\otimes_{\Lambda}\Gamma:\sttilt\Lambda\rightarrow\sttilt\Gamma$ given by $U\mapsto
U\otimes_{\Lambda}\Gamma$ is bijective.
\item The map in $(1)$ is an isomorphism of posets.
\item If $X$ is an indecomposable $\tau$-rigid $\Lambda$-module, then $X\otimes_\Lambda\Gamma$ is an indecomposable $\Gamma$-module.
\end{enumerate}
\end{theorem}

\begin{proof}
(1) For any $U\in\sttilt\Lambda$, there exists $T\in\langle
I_1,\ldots,I_{n-1}\rangle$ and $0\le i\le n$ such that
\[U=\mu_{[i+1,n]}(T)=T_1\oplus\cdots\oplus T_i\oplus\overline{T_i}\oplus\cdots\oplus\overline{T_{n-1}}\]
by Theorem \ref{4.14}. In this case, we have
\[U\otimes_\Lambda\Gamma=\left\{\begin{array}{cc}
(T_1/L_1)\oplus\cdots\oplus(T_i/L_i)\oplus\overline{T_i}\oplus\cdots\oplus\overline{T_{n-1}}&\mbox{if }\ i\ge1,\\
0\oplus\overline{T_1}\oplus\cdots\oplus\overline{T_{n-1}}&\mbox{if }\ i=0.\\
\end{array}\right.\]
For any $1\le j\le n$, $\overline{T_j}$ does not have $S_n$ as a
composition factor, and $T_j/L_{j}$ has $S_n$ as a composition
factor. Therefore the integer $i$ can be recovered from $U$ as the
number of indecomposable direct summands of $U$ which have $S_n$ as
a composition factor. Moreover, by Lemmas \ref{4.2}(2) and
\ref{4.3}(2), the socle of the $j$-th direct summand of
$U\otimes_\Lambda\Gamma$ is $S_{n-j+1}$ if $1\leq j\leq i$, and
either $0$ or $S_{n-j+1}$ if $i+1\leq j\leq n$.

Now assume that another $U'\in\sttilt\Lambda$ satisfies
$U\otimes_\Lambda\Gamma\cong U'\otimes_\Lambda\Gamma$, and take
$T'\in\langle I_1,\ldots,I_{n-1}\rangle$ and $1\le i'\le n$ such
that $U'=\mu_{[i'+1,n]}(T')$. By the argument above, we have $i=i'$.
By looking at the socle of each indecomposable direct summand, we
have $T_j/L_{j}\cong T'_j/L_j$ for any $1\le j\le i$ and
$\overline{T_j}\cong\overline{T'_j}$ for any $i\le j\le n-1$. They
imply $T_j\cong T'_j$ for any $1\le j\le n-1$ by Lemmas \ref{4.2}(3)
and \ref{4.3}(4). Since $T_n=P_n=T'_n$, we have $T\cong T'$ and
hence $U=\mu_{[i+1,n]}(T)\cong\mu_{[i+1,n]}(T')=U'$.

(3) By Theorem \ref{4.14}(3), $X$ has a form $T_i$ or
$\overline{T_i}$ for some $T\in\langle I_1,\ldots,I_{n-1}\rangle$
and $1\le i\le n$. Since $T_i\otimes_\Lambda\Gamma=T_i/L_i$ and
$\overline{T_i}\otimes_\Lambda\Gamma=\overline{T_i}$ are
indecomposable by Lemmas \ref{4.2}(2) and \ref{4.3}(2), the
assertion follows.

(2) The map $-\otimes_\Lambda\Gamma$ preserves mutations. In
fact, if $U=\mu_i(T)$ for $T,U\in\sttilt\Lambda$, then
$U\otimes_\Lambda\Gamma$ and $T\otimes_\Lambda\Gamma$ have the same
indecomposable direct summands except the $i$-th summand by (3) and
the injectivity of
$-\otimes_\Lambda\Gamma:\sttilt\Lambda\to\sttilt\Gamma$. Therefore
we have $U\otimes_\Lambda\Gamma=\mu_i(T\otimes_\Lambda\Gamma)$.

In particular, $-\otimes_\Lambda\Gamma$ gives an isomorphism ${\rm
H}(\sttilt\Lambda)\to{\rm H}(\sttilt\Gamma)$ of Hasse quivers
by Theorem \ref{2.m}.
Thus $-\otimes_\Lambda\Gamma:\sttilt\Lambda\to\sttilt\Gamma$ is an
isomorphism of posets by Lemma \ref{2.15}.
\end{proof}

\begin{remark}
Theorem \ref{5.4} gives another proof of Mizuno's result
\cite[Theorem 2.21]{M}.

On the other hand, we can give another shorter proof by using
Mizuno's result \cite[Theorem 2.21]{M}. By Proposition \ref{5.1}(3),
we have a surjective map
$-\otimes_\Lambda\Gamma:\sttilt\Lambda\to\sttilt\Gamma$. This must
be injective since we know $\#\sttilt\Lambda=(n+1)!=\#\sttilt\Gamma$
by Theorem \ref{4.13} and Mizuno's result.
\end{remark}

As a corollary, we get the following.

\begin{corollary}\label{5.6}
Let $\Lambda$ be the Auslander algebra of $K[x]/(x^n)$ and $\Gamma$
a preprojective algebra of Dynkin type $A_n$. There are isomorphisms
between the following posets:
\begin{enumerate}[\rm(1)]
\item The poset $\sttilt\Lambda$ with the generation order.
\item The poset $\sttilt\Gamma$ with the generation order.
\item The symmetric group $\mathfrak{S}_{n+1}$ with the opposite of the left order.
\item The poset $\sttilt(\Lambda^{\op})$ with the opposite of the generation order.
\item The poset $\sttilt(\Gamma^{\op})$ with the opposite of the generation order.
\item The symmetric group $\mathfrak{S}_{n+1}$ with the right order.
\end{enumerate}
\end{corollary}

\begin{proof}
The isomorphism from (1) to (2) given by $-\otimes_{\Lambda}\Gamma$
is showed in Theorem \ref{5.4}. The isomorphism from (3) to (1)
given by $I$ is showed in Theorems \ref{4.13} and \ref{4.16}. The
isomorphism between (1) and (4) (resp. (2) and (5)) is given in
\cite{AIR}.
\end{proof}

\begin{example} Denote by $\Gamma_n$ the preprojective algebra of type $A_n$. Then
\begin{enumerate}[\rm(1)]
 \item The Hasse quiver ${\rm H}(\sttilt\Gamma_2)$ is of the following form,
 where $\xrightarrow{i}$ shows $\mu_i$.
$$\begin{xy}
0;<3pt,0pt>:<0pt,2pt>::
(0,0)
*+{\left[\begin{smallmatrix} 1\\ &2\\ \end{smallmatrix}\middle|
\begin{smallmatrix} &2\\ 1\\ \end{smallmatrix}\right]}="1",
(-25,-4)
*+{\left[\begin{smallmatrix} 2\\ \end{smallmatrix}\middle|
\begin{smallmatrix} &2\\ 1\\ \end{smallmatrix}\right]}="2",
(25,-4)
*+{\left[\begin{smallmatrix}1\\&2\\ \end{smallmatrix}\middle|
\begin{smallmatrix} 1\\ \end{smallmatrix}\right]}="3",
(-25,-20)
*+{\left[\begin{smallmatrix} 2\\ \end{smallmatrix}\middle|\ \right]}="4",
(25,-20)
*+{\left[\ \middle|\begin{smallmatrix} 1\\ \end{smallmatrix}\right]}="5",
(0,-24)
*+{ \left[\ \middle|\ \right]}="6",
\ar"1";"2",|{1} \ar"1";"3",|{2} \ar"2";"4",|{2} \ar"3";"5",|{1} \ar"4";"6",|{1}
\ar"5";"6",|{2}
\end{xy}$$

\item The Hasse
quiver ${\rm H}(\sttilt\Gamma_3)$ is of the following form,
where $\xrightarrow{i}$ shows $\mu_i$.
$$\begin{xy}
0;<3.4pt,0pt>:<0pt,2.5pt>::
(0,0)
*+{ \left[\begin{smallmatrix} 1\\ &2\\
&&3\end{smallmatrix}\middle| \begin{smallmatrix} &2\\ 1&&3\\
&2\\ \end{smallmatrix}\middle| \begin{smallmatrix} &&3\\ &2\\
1\\ \end{smallmatrix}\right]}="1", (-30,-16)
*+{\left[\begin{smallmatrix} \\ 2\\
&3\end{smallmatrix}\middle| \begin{smallmatrix} &2\\ 1&&3\\
&2\\ \end{smallmatrix}\middle| \begin{smallmatrix} &&3\\ &2\\
1\\ \end{smallmatrix}\right]}="2", (0,-16)
*+{\left[\begin{smallmatrix} 1\\ &2\\
&&3\end{smallmatrix}\middle| \begin{smallmatrix} 1&&3\\
&2\\ \end{smallmatrix}\middle| \begin{smallmatrix} &&3\\ &2\\
1\\ \end{smallmatrix}\right]}="3", (30,-16)
*+{\left[\begin{smallmatrix} 1\\ &2\\
&&3\end{smallmatrix}\middle| \begin{smallmatrix} &2\\ 1&&3\\
&2\\ \end{smallmatrix}\middle| \begin{smallmatrix} &2\\
1\end{smallmatrix}\right]}="4", (-50,-32)
*+{\left[\begin{smallmatrix} 2\\ &3\end{smallmatrix}\middle|
\begin{smallmatrix} &3\\ 2\\ \end{smallmatrix}\middle|
\begin{smallmatrix} &&3\\ &2\\1\\ \end{smallmatrix}\right]}="5", (-25,-32)
*+{\left[\begin{smallmatrix} 3\end{smallmatrix}\middle|
\begin{smallmatrix} 1&&3\\ &2\\ \end{smallmatrix}\middle|
\begin{smallmatrix} &&3\\ &2\\
1\\ \end{smallmatrix}\right]}="6", (0,-32)
*+{\left[\begin{smallmatrix} \\ 2\\
&3\end{smallmatrix}\middle| \begin{smallmatrix} &2\\ 1&&3\\
&2\\ \end{smallmatrix}\middle| \begin{smallmatrix} &2\\
1\end{smallmatrix}\right]}="7", (25,-32)
*+{\left[\begin{smallmatrix} 1\\ &2\\
&&3\end{smallmatrix}\middle| \begin{smallmatrix} 1&&3\\
&2\\ \end{smallmatrix}\middle| \begin{smallmatrix}
1\end{smallmatrix}\right]}="8", (50,-32)
*+{\left[\begin{smallmatrix} 1\\ &2\\
&&3\end{smallmatrix}\middle| \begin{smallmatrix} 1\\
&2\end{smallmatrix}\middle| \begin{smallmatrix} &2\\
1\end{smallmatrix}\right]}="9", (-55,-48)
*+{\left[\begin{smallmatrix} 2\\ &3\end{smallmatrix}\middle|
\begin{smallmatrix} &3\\ 2\\ \end{smallmatrix}\middle|\ \right]}="10",
(-35,-48)
*+{\left[\begin{smallmatrix}
3\end{smallmatrix}\middle| \begin{smallmatrix} &3\\
2\\ \end{smallmatrix}\middle| \begin{smallmatrix} &&3\\
&2\\1\\ \end{smallmatrix}\right]}="11", (-10,-48)
*+{\left[\begin{smallmatrix}
3\end{smallmatrix}\middle| \begin{smallmatrix} 1&&3\\
&2\\ \end{smallmatrix}\middle| \begin{smallmatrix}
1\end{smallmatrix}\right]}="12", (10,-48)
*+{\left[\begin{smallmatrix} \\ 2\\
&3\end{smallmatrix}\middle| \begin{smallmatrix}
2\end{smallmatrix}\middle| \begin{smallmatrix} &2\\
1\end{smallmatrix}\right]}="13", (35,-48)
*+{\left[\begin{smallmatrix} 1\\ &2\\
&&3\end{smallmatrix}\middle| \begin{smallmatrix} 1\\
&2\end{smallmatrix}\middle| \begin{smallmatrix}
1\end{smallmatrix}\right]}="14", (55,-48)
*+{\left[\ \middle|\begin{smallmatrix} 1\\ &2\end{smallmatrix}\middle|
\begin{smallmatrix} &2\\ 1\end{smallmatrix}\right]}="15",
(-50,-64)
*+{\left[\begin{smallmatrix}
3\end{smallmatrix}\middle| \begin{smallmatrix} &3\\
2\\ \end{smallmatrix}\middle|\ \right]}="16", (-25,-64)
*+{\left[\begin{smallmatrix} 2\\
&3\end{smallmatrix}\middle| \begin{smallmatrix}
2\end{smallmatrix}\middle|\ \right]}="17", (-0,-64)
*+{\left[\begin{smallmatrix} 3\end{smallmatrix}\middle|\ \middle|
\begin{smallmatrix} 1\end{smallmatrix}\right]}="18", (25,-64)
*+{\left[\ \middle|\begin{smallmatrix}
2\end{smallmatrix}\middle| \begin{smallmatrix} &2\\
1\end{smallmatrix}\right]}="19", (50,-64)
*+{\left[\ \middle|\begin{smallmatrix} 1\\
&2\end{smallmatrix}\middle| \begin{smallmatrix}
1\end{smallmatrix}\right]}="20", (-30,-80)
*+{\left[\begin{smallmatrix}
3\end{smallmatrix}\middle|\ \middle|\ \right]}="21", (0,-80)
*+{\left[\ \middle|\begin{smallmatrix}
2\end{smallmatrix}\middle|\ \right]}="22", (30,-80)
*+{\left[\ \middle|\ \middle|\begin{smallmatrix}
1\end{smallmatrix}\right]}="23", (0,-96)
*+{\left[\ \middle|\ \middle|\ \right]}="24",  \ar"1";"2",|{1}
\ar"1";"3",|{2} \ar"1";"4",|{3} \ar"2";"5",|{2} \ar"2";"7",|{3} \ar"3";"6",|{1}
\ar"3";"8",|{3} \ar"4";"7",|{1} \ar"4";"9",|{2} \ar"5";"10",|{3} \ar"5";"11",|{1}
\ar"6";"11",|{2} \ar"6";"12",|{3} \ar"7";"13",|{2} \ar"8";"12",|{1} \ar"8";"14",|{2}
\ar"9";"14",|{3} \ar"9";"15",|{1} \ar"10";"16",|{1} \ar"10";"17",|{2} \ar"11";"16",|{3}
\ar"12";"18",|{2} \ar"13";"17",|{3} \ar"13";"19",|{1} \ar"14";"20",|{1}
\ar"15";"19",|{2} \ar"15";"20",|{3} \ar"16";"21",|{2} \ar"17";"22",|{1}
\ar"18";"21",|{3} \ar"18";"23",|{1} \ar"19";"22",|{3} \ar"20";"23",|{2}
\ar"21";"24",|{1} \ar"22";"24",|{2} \ar"23";"24",|{3} \end{xy}$$
\end{enumerate}
\end{example}


\begin{thebibliography}{101}

\bibitem[A1]{A1} T. Adachi, {\it The classification of $\tau$-tilting modules over Nakayama
algebras}, J. Algebra, 452(2016), 227-262.

\bibitem[A2]{A2} T. Adachi, {\it Characterizing $\tau$-rigid-finite algebras with radical
square zero,} Proc. Amer. Math. Soc., 144(11)(2016), 4673-4685.

\bibitem[AAC]{AAC} T. Adachi, T. Aihara and A. Chan, {\it Classification of two-term tilting complexes over Brauer graph algebras,} Math. Z., 290(2)(2018),1-36.

\bibitem[AIR]{AIR} T. Adachi, O. Iyama and I. Reiten, {\it $\tau$-tilting
theory}, Compos. Math., {150}(3)(2014), 415-452.

\bibitem [AiI]{AiI} T. Aihara and O. Iyama, {\it Silting mutation in
triangulated categories}, J. Lond. Math. Soc., {85}(3)(2012),
633-668.

\bibitem [AnHK]{AnHK} L. Angeleri H\"{u}gel, D, Happel and H. Krause,
{\it Handbook of tilting modules}, London Math. Soc. Lecture Notes
Series, {332}.

\bibitem [AnMV]{AnMV} L. Angeleri H\"{u}gel, F. Marks, J. Vit\'oria, {\it Silting
modules}, Int. Math. Res. Not., 4(2016), 1251-1284.

\bibitem[AsSS]{AsSS} I. Assem, D. Simson and A. Skowro\'nski,
{\it Elements of the Representation Theory of Associative Algebras.
Vol. 1. Techniques of Reperesentation Theory,} London Math. Soc.
Student Texts, {65}, Cambridge Univ. Press, Cambridge, 2006.

\bibitem[AuB]{AuB} M. Auslander, M. Bridger, {\it Stable module theory}, Memoirs
Amer. Math. Soc., No. {94}, American Mathematical Society,
Providence, R.I. 1969, 146.

\bibitem [AuPR] {AuPR} M. Auslander, M. I. Platzeck and I. Reiten, {\it Coxeter functions without
diagrams,} Trans. Amer. Math. Soc., {250}(1979), 1-12.

\bibitem[AuR] {AuR} M. Auslander and I. Reiten, {\it Cohen-Macaulay and Gorenstein algebras.} In: Representation theory of finite
groups and finite-dimensional algebras, Bielefeld, 1991, edited by
G. O. Michler and C. M. Ringel, Progr. Math., Vol. 95, Birkhauser,
Basel, 1991, 221-245.

\bibitem[AuRS]{AuRS} M. Auslander, I. Reiten and S. O. Smal$\phi$, {\it Representation Theory of Artin
Algebras}, Cambridge Studies in Advanced Mathematics, 36. Cambridge
University Press, Cambridge, 1997.


\bibitem[AuS]{AuS} M. Auslander and S. O. Smal$\phi$, {\it Almost split sequences in
subcategories,} J. Algebra, {69}(1981), 426-454, Addendum; J.
Algebra, {71}(1981), 592-597.

\bibitem[B]{B} K. Bongartz, {\it Tilted Algebras}, Proc. ICRA III (Puebla 1980), Lecture Notes in Math. No.{ 903}, Springer-Verlag
1981, 26-38.

\bibitem[BGP]{BGP} I. N. Bernstein, I. M. Gelfand and V. A. Ponomarev, {\it Coxeter functors and Gabriel's theorem}, Russ. Math.
Surv. {28}(1973), 17-32.

\bibitem[BjB]{BjB} A. Bjorner and F. Brenti, {\it Combinatorics of Coxeter groups: Graduate
Texts in Mathematics}, 231. Springer, New York, 2005.

\bibitem[BrB]{BrB} S. Brenner and M. C. R. Butler, Generalization of the
Bernstein-Gelfand-Ponomarev reflection functors, Lecture Notes in
Math. 839, Springer-Verlag (1980), 103-169.

\bibitem[BHRR]{BHRR} T. Br\"{u}stle, L. Hille, C. M. Ringel and G. R\"{o}hrle.
 {\it The $\Delta$-Filtered Modules Without Self-Extensions for
the Auslander Algebra of $k[T]/ \langle T^{n} \rangle$,}
 Algebr. Represent. Theory, {2}(1999), 295-312.

\bibitem[BIRS]{BIRS} A. B. Buan, O. Iyama, I. Reiten and J. Scott, {\it Cluster structures for
2-Calabi-Yau categories and unipotent groups}. Compos. Math., {145}
(4)(2009), 1035-1079.

\bibitem[BMRRT]{BMRRT} A. B. Buan, R. Marsh, M. Reineke, I. Reiten and G.
Todorov, {\it Tilting theory and cluster combinatorics}, Adv. Math.
204(2)(2006), 572-618.

\bibitem[DIJ]{DIJ} L. Demonet, O. Iyama and G. Jasso, {\it $\tau$-tilting finite algebras, bricks and $g$-vectors}, Int. Math. Res. Not., 3(2019), 852-892.

\bibitem[DIRRT]{DIRRT} L. Demonet, O. Iyama, N. Reading, I. Reiten and H. Thomas, {\it Lattice theory of torsion classes}, arXiv:1711.01785.

\bibitem[DF]{DF} H. Derksen and J. Fei, {\it General presentations of algebras}, Adv. Math., 278(2015), 210-237.

\bibitem[Ha]{Ha} D. Happel, {\it Triangulated categories in the representation theory
of finite-dimensional algebras,} London Mathematical Society Lecture
Note Series, 119. Cambridge University Press, Cambridge, 1988.

\bibitem[HaR]{HaR} D. Happel and C. M. Ringel, {\it Tilted algebras,}
Trans. Amer. Math. Soc., 274(2)(1982), 399-443.

\bibitem[HaU]{HaU} D. Happel and L. Unger, {\it On a partial order of tilting modules,}  Algebras
and Representation Theory, {8}(2005), 147-156.

\bibitem[HKM]{HKM} M. Hoshino, Y. Kato, J. Miyachi, {\it On $t$-structures and torsion theories induced by compact objects}, J. Pure Appl. Algebra, 167(2002), no. 1, 15-35.

\bibitem[HuZ]{HuZ} Z. Huang and Y. Zhang, {\it G-stable support $\tau$-tilting
modules,} Front. Math. China, 11(4)(2016), 1057-1077.

\bibitem[IJY]{IJY} O. Iyama, P. Jorgensen and D. Yang, {\it Intermediate
co-t-structures, two-term silting objects, $\tau$-tilting modules
and torsion classes,} Algebra Number Theory, {8}(10)(2014),
2413-2431.

\bibitem[IR]{IR} O. Iyama and I. Reiten, {\it Fomin-Zelevinsky mutation and tilting modules
over Calabi-Yau algebras}, Amer. J. Math., 130(4)(2008), 1087-1149.

\bibitem [IRRT]{IRRT} O. Iyama, N. Reading, I. Reiten and H. Thomas, {\it Lattice structure of Weyl groups via representation theory of preprojective algebras,}  Compos. Math., 154(6)(2018), 1269-1305.


\bibitem[IY]{IY} O. Iyama and Y. Yoshino, {\it Mutations in triangulated categories and rigid Cohen-Macaulay
modules,} Invent. Math., {172}(2008), 117-168.

\bibitem[IZ]{IZ} O. Iyama and X. Zhang, {\it Tilting modules over Auslander-Gorenstein algebras}, to appear in Pacific J. Math, arXiv:1801.04738.

\bibitem[KR]{KR} B. Keller and I. Reiten, {\it Cluster-tilted algebras are Gorenstein
and stably Calabi-Yau,} Adv. Math., {211}(1)(2007), 123-151.

\bibitem[KV]{KV} B. Keller and D. Vossieck, {\it Aisles in derived categories},  Bull. Soc. Math. Belg. S\'er. A 40(2)(1988), 239-253.

\bibitem[J]{J} G. Jasso, {\it Reduction of $\tau$-Tilting Modules and Torsion Pairs,} Int. Math. Res. Not.,
{16}(2015), 7190-7237.

\bibitem[M]{M} Y. Mizuno, {\it Classifying $\tau$-tilting modules over preprojective algebras
of Dynkin type,} Math. Zeit., {277}(3)(2014), 665-690.

\bibitem[RS]{RS} C. Riedtmann, A. Schofield, {\it On a simplicial complex associated
with tilting modules,} Comment. Math. Helv., {66}(1)(1991), 70-78.

\bibitem[T]{T} Y. Tsujioka, {\it Tilting modules over the Auslander algebra of $K[x]/(x^n)$,}
Master Thesis in Graduate School of Mathematics in Nagoya
University, 2008.

\bibitem[W]{W} J. Wei, {\it $\tau$-tilting modules and $*$-modules},
J. Algebra, 414(2014), 1-5.

\bibitem[Z1]{Z1} X. Zhang, {\it $\tau$-rigid modules for algebras with radical square
zero,} arXiv:1211.5622.

\bibitem[Z2]{Z2} X. Zhang, {\it $\tau$-rigid modules over Auslander algebras,} Taiwanese J. Math,
21(4)(2017), 327-338.

\bibitem[Zh]{Zh} Y. Zhang, {\it On mutation of $\tau$-tilting
modules,} Comm. Algebra, 45(6)(2017), 2726-2729.




\end{thebibliography}
\end{document}